\documentclass[11pt]{amsart}%
\usepackage{amsmath}%
\usepackage{amssymb}%
\usepackage{amscd}%
\usepackage{color}%
\usepackage{mathrsfs}%
\usepackage[all]{xy}%
\def\ol#1{\overline{#1}}
\def\wh#1{\widehat{#1}}
\def\wt#1{\widetilde{#1}}
\def\ul#1{\underline{#1}}
\def\smcompactification#1{\ol{#1}}
\theoremstyle{plain}
    \newtheorem{theorem}{Theorem}[section]
    
    \newtheorem{proposition}[theorem]{Proposition}
    \newtheorem{lemma}[theorem]{Lemma}
    \newtheorem{corollary}[theorem]{Corollary}
\theoremstyle{definition}
    \newtheorem{definition}[theorem]{Definition}

    \newtheorem{remark}[theorem]{Remark}
\def\Alphabet{A,B,C,D,E,F,G,H,I,J,K,L,M,N,O,P,Q,R,S,T,U,V,W,X,Y,Z}
\def\alphabet{a,b,c,d,e,f,g,h,i,j,k,l,m,n,o,p,q,r,s,t,u,v,w,x,y,z}
\def\endpiece{xxx}
\def\makeAlphabet[#1]{\expandafter\makeA#1,xxx,}
\def\makealphabet[#1]{\expandafter\makea#1,xxx,}
\def\makeA#1,{\def\temp{#1}\ifx\temp\endpiece\else%
\mkbb{#1}\mkfrak{#1}\mkbf{#1}\mkcal{#1}\mkscr{#1}\expandafter\makeA\fi}%
\def\makea#1,{\def\temp{#1}\ifx\temp\endpiece\else\mkfrak{#1}\mkbf{#1}\expandafter\makea\fi}%
\def\mkbb#1{\expandafter\def\csname bb#1\endcsname{\mathbb{#1}}}
\def\mkfrak#1{\expandafter\def\csname fr#1\endcsname{\mathfrak{#1}}}
\def\mkbf#1{\expandafter\def\csname b#1\endcsname{\mathbf{#1}}}
\def\mkcal#1{\expandafter\def\csname c#1\endcsname{\mathcal{#1}}}
\def\mkscr#1{\expandafter\def\csname s#1\endcsname{\mathscr{#1}}}
\def\makeop[#1]{\xmakeop#1,xxx,}
\def\mkop#1{\expandafter\def\csname #1\endcsname{{\mathrm{#1}}}} %
\def\xmakeop#1,{\def\temp{#1}\ifx\temp\endpiece\else\mkop{#1}\expandafter\xmakeop\fi}%
\makeAlphabet[\Alphabet]
\makealphabet[\alphabet]
\makeop[Alt,End,Ext,Hom,Sym,Tor]
\makeop[CH,Pic,div]
\makeop[Ker,Im,Coker,Coim,id,pr]
\makeop[Spec,Spf,Spm,Isom]
\makeop[Re,Im]%
\makeop[dR,Nis,crys,rig,syn,tor,Cont]
\makeop[Gal,Res,ab,res]
\makeop[pol]
\makeop[Var,an,Isoc,univ,Cusp]%
\makeop[Lie,coLie,MIC,DR,Gr,Ab,Cone,ord,Frob]
\makeop[Meas,arith,mot,can,tors,conv]
\makeop[Gl,Sl]
\makeop[Eis,tr,Fil,gr,et,sp,Tate]
\makeop[Gl,Sl,GL,SL]

\def\Gm{{\bbG_m}}
\def\isom{\cong}
\def\verk{\circ}

\def\Isocda{\Isoc^{\kern-0.5mm\dagger}}

\def\shM{\cM}
\def\isom{\cong}
\def\verk{\circ}

\def\Garith{\Gamma^{\arith}}
\def\GarithN{\Gamma(N)^{\arith}}
\def\shomega{\underline{\omega}}

\def\MN{M(N)}
\def\MarithN{M_\arith(N)}

\def\sElliptic{E}
\def\sModular{M}

\def\et{\operatorname{\text{\'e}t}}
\def\frobphi{\phi}

\begin{document}
\title[Elliptic Polylog]{$p$-adic elliptic polylogarithm, $p$-adic Eisenstein series and Katz measure}
\begin{abstract}
	The specializations of the motivic elliptic polylogarithm on the universal elliptic curve to the modular curve
	are referred to as Eisenstein classes.  In this paper, we prove that the syntomic realizations of the Eisenstein classes 
	restricted to the ordinary locus of  the modular curve may be expressed using $p$-adic Eisenstein-Kronecker series, which are
	$p$-adic modular forms defined using the two-variable $p$-adic measure with values in $p$-adic 
	modular forms constructed by Katz.	
\end{abstract}
\author{Kenichi Bannai and Guido Kings}
\date{\today}
\maketitle

\setcounter{section}{-1}

%
%
%
%
\section{Introduction}
%
%
%
%

The elliptic polylogarithm was introduced by Beilinson and Levin in their seminal
paper \cite{BL}. There the elliptic polylogarithm appears as an element in the motivic 
cohomology of a relative elliptic curve $\pi:E\to B$ minus the zero section. 
The specializations of this element along torsion sections of the relative elliptic curve are what is known as 
\textit{Eisenstein classes}, and in their paper, they explicitly described these classes in terms of certain real analytic 
Eisenstein-Kronecker series. 

These Eisenstein classes have found numerous applications to special values of $L$-functions.
Immediate is the relation to $L$-values of CM-elliptic curves as considered by Deninger
\cite{Den} and in \cite{Ki}. Less immediate, but implicit in earlier work of Beilinson \cite{Bei},
is the relation to $L$-values of modular cusp forms, where one has to consider 
cup-products of these Eisenstein classes. This work by Beilinson has in turn
inspired Kato's work on the Birch-Swinnerton-Dyer conjecture. An application to non-critical
values of Dirichlet series is given in \cite{HK1}.

For the application to finer integrality questions about $L$-values, it is
necessary to investigate not only the Hodge realization but also the \'etale and $p$-adic realizations
of these classes. In the \'etale situation, as shown in \cite{Ki},
the elliptic polylogarithm for relative elliptic curves may be described essentially
as the Kummer map of elliptic units on the modular curve.
In the $p$-adic, or more precisely in the syntomic case, the only result known so far is the case for
single elliptic curves with complex multiplication which have good ordinary reduction at $p \geq 5$ \cite{Ba3} \cite{BKT}.
Up until now, there has not been any research concerning the relative case.

In this paper we consider the moduli space of  elliptic curves and the 
specialization of the elliptic polylogarithm of the universal elliptic curve at torsion points. 
We define the \textit{syntomic Eisenstein class} to be the rigid syntomic realization of such specializations.
Our main result, Theorem \ref{thm: main}, expresses the restriction of these classes to the ordinary locus
of the moduli space in terms of $p$-adic Eisenstein-Kronecker series, which are defined explicitly using the
$p$-adic Eisenstein measure constructed following Katz.

In complete analogy with the case for absolute Hodge cohomology,
rigid syntomic cohomology $H^{1}_\syn( \sX, \cM)$ with values in an admissible filtered
overconvergent $F$-isocrystal $\cM = (M, \nabla, F, \Phi)$ satisfying $F^0 M = 0$
may be described by pairs $(\alpha,\xi)$, where
$\alpha$ is an overconvergent section in $\Gamma(\cX_K, M_\rig)$ and $\xi$ is 
an algebraic differential form satisfying the differential equation
\begin{equation}\label{eq: alpha}
	\nabla(\alpha)=(1 - \Phi) \xi,
\end{equation}
where $\Phi$ is the Frobenius on $\cM$ (see Proposition \ref{prop: fund class} for details).
The main problem is to explicitly describe the pair $(\alpha, \xi)$ corresponding to
the syntomic Eisenstein class.

In the case of the syntomic Eisenstein class, $\xi$ is the differential form corresponding to
the de Rham Eisenstein class.  In order to determine the syntomic class, it is necessary to find
a section $\alpha$ satisfying the differential equation \eqref{eq: alpha}.
The main idea of this paper, which makes the computation elegant, is not to solve this
equation directly but to translate it to the measure space used by Katz. It is one of the main insights of
Katz, that the Gauss-Manin connection $\nabla$ acts on this measure space just as a twist by a certain
character. This insight allows us to reformulate the above differential equation into an equation on this
measure space, which may be solved explicitly.

In \cite{BL}, Beilinson and Levin constructs certain two-variable $p$-adic measures using the \'etale
realization of the elliptic polylogarithm, and in \cite{BL} 2.5.12, asks if there is a relation between
this measure and Katz's theory of $p$-adic Eisenstein series.  The syntomic Eisenstein class corresponds
to the \'etale Eisenstein class via $p$-adic Hodge theory, and our main result relating the syntomic class
to functions constructed using Katz's $p$-adic Eisenstein measure seems to be an important step towards 
answering this question.

Furthermore, since we are using $p$-adic modular forms, we may only describe the Eisenstein class on the ordinary
locus.  However, the syntomic Eisenstein class itself is defined on the entire modular curve, including over
the supersingular disc.  Explicit description of these classes over supersingular points as well as 
ramifications to the study of $p$-adic modular forms are interesting topics for further investigation.

Let us give an overview of the sections in this paper. In the first section, we review the definition of the 
Eisenstein classes in motivic and de Rham cohomology as specializations of the elliptic polylogarithm. 
In the second section, we define the Eisenstein class in rigid syntomic cohomology. In the third section, we write down 
explicit formulas for the de Rham realizations of the Eisenstein classes.  In the fourth section, we
review the modular curve parameterizing elliptic curves with trivializations, and use this curve
to calculate syntomic cohomology on the ordinary locus of the elliptic curve.
In the final section, we construct $p$-adic Eisenstein-Kronecker series 
using the $p$-adic measure of Katz.  We then use these series to describe the $p$-adic Eisenstein classes.
In the appendix, we review the theory of rigid syntomic cohomology with coefficients.

\setcounter{tocdepth}{1}
\tableofcontents

%
%
\section{Polylog and Eisenstein classes}
%
%

%
\subsection{Moduli spaces}
%

Let $N\ge 1$ be an integer.  In his paper \cite{Ka3}, Katz works systematically with
$\GarithN$-structures. Let $B$ be a ring and $E/B$ be an elliptic curve.
Consider the Weil pairing
$$
	e_N:E[N]\times E[N]\to \mu_N.
$$
A $\GarithN$-level structure of $E$ is defined to be an isomorphism
$$
	\beta:\mu_N\times \bbZ/N\bbZ\isom E[N],
$$
where $\mu_N$ is the group scheme of $N$-th roots of unity and $E[N]$ the $N$-torsion 
points of $E$, such that the Weil pairing becomes under $\beta$ the standard pairing
$$
	<(\zeta_1,n),(\zeta_2,m)>=\zeta_1^m/\zeta_2^n.
$$
Note that for $N\ge 3$ the functor ``isomorphism classes of
$\GarithN$-elliptic curves $(E,\beta)$" is represented by a smooth affine curve
$\MarithN$ over $\bbZ$ with geometrically irreducible fibers. For any ring $B$ we let
$$
	\MarithN_B:=\MarithN\otimes_{\bbZ}B
$$
be the base change of $\MarithN$.

Let us explain the connection of the $\GarithN$-structures with the more usual $\Gamma(N)$-level structure. This is an isomorphism
$$
	\alpha:\bbZ/N\bbZ\times\bbZ/N\bbZ\isom E[N],
$$
which can only exist if $N$ is invertible on $B$.
The functor "isomorphism classes of
$\Gamma(N)$-elliptic curves $(E,\alpha)$" is for $N\ge 3$ represented by the smooth 
affine curve 
$$
	\MN\isom \MarithN  \otimes_{\bbZ} \bbZ[1/N,\zeta_N],
$$
where $\zeta_N\in\mu_N$ is a primitive $N$-th root of unity. For any $\bbZ[1/N,\zeta_N]$-algebra $B$, 
we let $\MN_B$ be the base change to $B$.

Let us make the relation between $\MN$ and $\MarithN$ more explicit. If $N$ is invertible on $B$, each $\Gamma(N)$-level
structure $\alpha$ on $E/B$ gives rise to a primitive $N$-th root of unity $\det(\alpha)$
and a $\GarithN$-level structure $\beta$ defined by
$\beta(\det(\alpha)^n,m):=\alpha(n,m)$. This correspondence establishes a
bijection between the set of $\Gamma(N)$-structures on $E/B$ and the set
of pairs $(\zeta_N,\beta)$, where $\zeta_N$ is a primitive $N$-th root of
unity and $\beta$ a $\GarithN$-structure on $E/B$.

Finally, we recall the action of $\GL_2(\bbZ/N\bbZ)$ on $M(N)$. An element
$\gamma\in \GL_2(\bbZ/N\bbZ)$ acts on $\MN$ from the right as follows:
$$
	(E/B,\alpha)\mapsto (E/B,\alpha\verk\gamma).
$$

%
\subsection{The elliptic polylogarithm and the Eisenstein classes in motivic cohomology}\label{motivicEis}
%

In this section we consider the situation where we have an elliptic curve $\pi:E\to M$ over the base scheme
$M$. In the application this will be the universal elliptic 
curve over the moduli schemes $\MN$ or $\MarithN$. We will use the elliptic polylogarithm 
in this situation to define Eisenstein classes in motivic cohomology.

Let $\sElliptic^k=\sElliptic\times_{\sModular}\dots\times_{\sModular} \sElliptic$ be
the $k$-fold relative fiber product. 
On $\sElliptic^k$ we have an operation of the semi-direct
product $\mu_2^k\rtimes\frS^k$ of the $k$-fold product of $\mu_2$ with
the symmetric group in $k$ letters on  $\sElliptic^k$. Following Scholl \cite{Sch}, denote by
$\varepsilon$ the character $\varepsilon\colon\mu_2^k\rtimes\frS^k\to\mu_2$,
which is the multiplication on $\mu_2^k$ and
the sign-character on $\frS^k$.

Let  $H^{k+1}_\mot(\sElliptic^k,\bbQ(k+1))(\varepsilon)$ be the $\varepsilon$-eigen part of the motivic 
cohomology group of $H^{k+1}_\mot(\sElliptic^k,\bbQ(k+1))$.
If we suppose the existence of an adequate theory of motivic sheaves, then we would have
\begin{align*}
	H^{k+1}_\mot(\sElliptic^k,\bbQ(k+1))(\varepsilon) &= H^{1}_\mot(\sModular,\Sym^k \sH(1))\\
	&= \Ext^1_{\sModular,\mot}(\bbQ(0), \Sym^k \sH(1)),
\end{align*}
where $\sH = \bbR^1 \pi_* \bbQ(1)$.  Lacking such a theory, we will use the left hand side of 
the above equality to play the role of $H^{1}_\mot(\sModular,\Sym^k \sH(1))$.

Recall from \cite{BL} 6.4.3. that for each non-zero torsion  point $t\in\sElliptic(M)$ the motivic elliptic polylog gives a class 
$$
	t^*\pol^{k+1}_\mot\in H^{k+1}_\mot(\sElliptic^k,\bbQ(k+1))(\varepsilon).
$$
\begin{definition}
	Let $\varphi=\sum a_tt$ be a formal linear combination of non zero torsion sections $t\in\sElliptic_\tors(M)$ with 
	coefficients in $a_t\in \bbQ$, then we define the \emph{motivic Eisenstein class}
	to be
	$$
		\Eis^{k+2}_\mot(\varphi):=\sum_{t\in E[N]\setminus \{0\}}a_tt^*\pol^{k+1}_\mot\in 
		H^{k+1}_\mot(\sElliptic^k,\bbQ(k+1))(\varepsilon).
	$$
\end{definition}
Recall also from \cite{BL} 1.3.13 that $\pol^{k+1}_\mot$ is compatible with base change, hence the 
motivic Eisenstein class is also compatible with base change.

Besides the motivic Eisenstein class we will use also realizations of the Eisenstein class
in other cohomology theories. We intend no general theory, but make a simple definition in
the cases of interest to us.
\begin{definition}\label{def:motivicEisenstein}
	Let $?=\dR,\rig,\syn, \operatorname{\text{\'e}t}$ and consider the regulator map
	$$
		r_?:H^{k+1}_\mot(\sElliptic^k,\bbQ(k+1))(\varepsilon)\to H^{k+1}_?(\sElliptic^k,\bbQ(k+1))(\varepsilon),
	$$
	then the image of $\Eis^{k+2}_\mot(\varphi)$ under $r_?$ is called the \emph{Eisenstein class in $?$-cohomology}.
\end{definition}

%
\subsection{The residues of the motivic Eisenstein classes at the cusps}
%

To give explicit formulas for the Eisenstein class in de Rham cohomology we need a
formula for the residues of these classes at the cusp. The easiest thing is to give this
formula in motivic cohomology. Using the compatibility of the regulator with the residue
map, gives then the formula in any cohomology theory we use.

In this section we let $N\ge 3$ and work with the $\Gamma(N)$-moduli scheme $M=\MN$. 
The formula we are after is due to Beilinson and
Levin \cite{BL} 2.4.7. We follow the exposition of \cite{HK1}.

Let $\ol\sModular$ be the compactification of $\sModular$ and $\smcompactification{\sElliptic}$ 
the N\'eron model of  $\sElliptic$ over $\ol\sModular$ and $\smcompactification{\sElliptic}^0$ 
its connected component. 
Let $\Cusp=\ol\sModular\setminus\sModular$ be the subscheme of cusps. 
The standard $N$-gon over $\Spec\,\bbZ[1/N, \zeta_N]$ with level $N$-structure
$\bbZ/N\times \bbZ/N\to \Gm\times\bbZ/N$ via 
$(a,b)\mapsto (\zeta_N^a,b)$
{\em defines} a section $\infty\colon\Spec\,\bbZ[1/N, \zeta_N]\to \Cusp$. 
We have a diagram
\begin{equation}\label{compdiag}
	\begin{CD} 
		\sElliptic @>j>> \smcompactification{\sElliptic} @<<< \smcompactification{\sElliptic}_{\Cusp}\\
		@V\pi VV @V\smcompactification{\pi} VV @V\smcompactification{\pi} VV\\
		\sModular@>j'>> \ol\sModular @<<< \Cusp\ .
	\end{CD}
\end{equation}
As in \cite{HK1} 1.1. we define the $\mu_2$-torsor 
\begin{equation}\label{isom-defn}
	\Isom=\Isom(\Gm,\tilde{E}^0_{\Cusp})
\end{equation}
on $\Cusp$.
Over $\infty$, we have a canonical trivialization $\Isom_\infty=\mu_{2,\infty}$ 
by the very definition of $\infty$.

As in loc. cit. the localization sequence
induces a $ GL_2(\bbZ/N)$-equivariant map 
$$
	\res^k\colon H^{k+1}_\mot(\sElliptic^k,\bbQ(k+1))(\varepsilon)\to H_\mot^0(\Isom,\bbQ(0)),
$$
the {\em residue map}. The image of $\res^k$ lies in the $(-1)^k$ eigenspace of
the $\mu_2$-action.
Let ${P}:=\left\{\left(\begin{smallmatrix} \ast&\ast\\ 0&1\end{smallmatrix}\right)\right\}\subset  GL_2$, then
$$
	\Isom\isom \coprod_{P(\bbZ/N)\setminus GL_2(\bbZ/N)}\Spec\,\bbZ[1/N, \zeta_N],
$$
where $\id\in  GL_2(\bbZ/N)$ corresponds to the section $1\in\mu_{2,\infty}=\Isom_\infty$. The right
action of $ GL_2(\bbZ/N)$ on $M(N)$ extends to an action on $\Isom$ by right multiplication on
$P(\bbZ/N)\setminus GL_2(\bbZ/N)$.
Still following \cite{HK1}, we define 
\begin{multline}
	\bbQ[\Isom]^{(k)}=
	\{h \colon GL_2(\bbZ/N)\to \bbQ \mid h(ug)=h(g)\text{ for }u\in P(\bbZ/N)\\ 
	\text{and } h(-\id \, g)= (-1)^k h(g)\}
\end{multline}
the space of formal linear combination of  points of $\Isom$ on which
$\mu_2$ operates by $(-1)^k$. The group $ GL_2(\bbZ/N)$ acts on this space in the
usual way from the left by $gh(x):=h(xg)$.
Obviously, we have 
$$
	H^0_\mot(\Isom,\bbQ(0))^{(k)}= \bbQ[\Isom]^{(k)}.
$$
To compute the residue map for Eisenstein series, we need also
$$
	\bbC[(\bbZ/N\bbZ)^2]:=\{\varphi:(\bbZ/N\bbZ)^2\to \bbC\}
$$
the space of $\bbC$-valued functions. We follow the convention in
\cite{HK2} and define the left $\GL_2(\bbZ/N)$-action
by $g\varphi(x):=\varphi(g^{-1}x)$. From now on we use $\alpha$ to identify 
$$
	\alpha:(\bbZ/N\bbZ)^2\isom E[N]
$$
so that we consider functions in $\bbC[(\bbZ/N\bbZ)^2]$ as linear combinations
of torsion sections.

The calculation of the residues of the polylog may be 
formulated using the horospherical map. For this we need some notions
about $L$-functions and finite Fourier transforms.
We define, following \cite{Ka3} and \cite{HK2}, for any $\varphi\in \bbC[(\bbZ/N\bbZ)^2]$ 
two partial Fourier transforms
\begin{align}\label{partialFourier}
	P_1\varphi(m,n)&:=\sum_v\varphi(v,n)e^{2\pi i mv/N}\\
\nonumber	P_2\varphi(m,n)&:=\sum_v\varphi(m,v)e^{2\pi i nv/N}
\end{align}
and the symplectic Fourier transform
\begin{equation}\label{symplecticFourier}
	\wh\varphi(m,n):=\frac{1}{N}\sum_{u,v}\varphi(u,v)e^{2\pi i (un-mv)/N}.
\end{equation}
We let $\varphi^t(m,n):=\varphi(n,m)$ and one has the relations 
$$
	P_2(\wh\varphi^t)=P_1(\varphi)
$$ and $P_2(\varphi^t)=P_1(\varphi)^t$.
For each $\varphi\in \bbC[(\bbZ/N\bbZ)^2]$ we also define its \emph{$L$-series}
\begin{equation}\label{L-defn}
	L(\varphi,s):=\sum_{m\ge 1}\frac{\varphi(m,0)}{m^s}.
\end{equation}
This $L$-series converges for $\Re s>1$ and has a meromorphic continuation to
$\bbC$, which satisfies the functional equation
$$
	L(P_2(\varphi),1-k)=\frac{(-1)^k2N^k(k-1)!}{(2\pi i)^k}L(\varphi,k).
$$	
\begin{definition}\label{horospherical}
	The \emph{horospherical} map is the $ GL_2(\bbZ/N)$-equivariant
	map  
	$$
		\rho^k\colon\bbQ[(\bbZ/N\bbZ)^2]\to \bbQ[\Isom]^{(k)},
	$$
	which maps a function $\varphi:(\bbZ/N\bbZ)^2\to \bbQ$ to the function
	\begin{align*}
		\rho^k(\varphi)(g)&:=\frac{N^k}{k!(k+2)}\sum_{t=(t_1,t_2)\in (\bbZ/N)^2}\varphi(g^{-1}t)B_{k+2}\left(\frac{t_2}{N}\right)\\
			&=\frac{-1}{Nk!}L(P_1(g\varphi),-k-1).
	\end{align*}
	Here 
	$B_{k+2}(t_2/N)$ is the Bernoulli polynomial
	evaluated at the representative of $t_2/N\in \bbR/\bbZ$ in $[0,1)$ and the last equation
	follows from \cite{HK2} p. 333 using that $P_1(g\varphi)=P_2(\wh{g\varphi}^t)$.
\end{definition}
The following proposition is due to Beilinson-Levin and is crucial
for the connection of the elliptic polylog to Eisenstein series. We consider the residue map as
$$
	\res^k\colon H^{k+1}_\mot(\sElliptic^k,\bbQ(k+1))(\varepsilon)\to \bbQ[\Isom]^{(k)}.
$$
\begin{proposition}[\cite{BL} 2.2.3., \cite{HK1} C.1.1.]\label{respol}
	Let $M=\MN$, $k\ge 0$ and $\varphi\in \bbQ[(\bbZ/N\bbZ)^2\setminus \{0\}]$ considered as a formal linear combination of non-zero $N$-torsion sections with coefficients
	in $\bbQ$. Then, for $g\in\GL_2(\bbZ/N)$,
	$$
		\res^k(\Eis^{k+2}_\mot(\varphi))(g)=\frac{-1}{N^{k-1}}\rho^k(\varphi)(g)
	$$
	where $\rho^k$ is the horospherical map.
\end{proposition}
\begin{proof}
	This theorem is proved in \cite{HK1} in \'etale cohomology. The above statement follows from the commutative 
	diagram for the \'etale regulator
	$$
		\begin{CD}
			H^{k+1}_\mot(\sElliptic^k,\bbQ(k+1))(\varepsilon)@>\res^k>> H^0_\mot(\Isom,\bbQ(0))^{(k)}\\
			@Vr_{\et}VV@V\isom Vr_{\et} V\\
			H^{k+1}_{\et}(\sElliptic^k,\bbQ_l(k+1))(\varepsilon)@>\res^k>> H^0_{\et}(\Isom,\bbQ_l(0))^{(k)}
		\end{CD}
	$$
	and the fact that the right vertical arrow is an isomorphism $\otimes \bbQ_l$.
\end{proof}

%
\subsection{Eisenstein classes in de Rham cohomology}
%

In this section we let $M=\MN$ and $\pi:E\to M$ the universal elliptic curve.

We define $\sH$ to be the coherent module with connection on $M$ defined as the higher direct image 
$$
	\sH = R^1 \pi_{\dR *} \cO_{E} := R^1\pi_* \left[ \cO_{E} \xrightarrow d \Omega^1_{E/M} \right]
$$
with the Gauss-Manin connection $\nabla:\sH\to \sH\otimes\Omega^1_M$. As usual we also define the coherent subsheaf of $\sH$
\begin{equation}\label{ulomegadefn}
	\ul\omega:=\pi_*\Omega^1_{E/M}.
\end{equation}
Then the natural inclusion $\ul\omega^{\otimes k} \hookrightarrow \Sym^k \sH$ defines a map
$$
	\Gamma( M,       \ul\omega^{\otimes k} \otimes \Omega^1_{M} ) 
	\hookrightarrow H^1_\dR( M, \Sym^k \sH),
$$
whose image defines the first Hodge filtration $F^1$ on $H^1_\dR( M, \Sym^k \sH)$.
If the scheme $M$ is $\MN_{\bbQ}$, the diagram (\ref{compdiag}) and the projector $\varepsilon$ define a localization 
sequence in de Rham cohomology:
$$
0 \rightarrow H_\dR^{k+1}(\ol E^k)(\varepsilon) \rightarrow H^{k+1}_\dR(E^k)(\varepsilon) 
	\rightarrow H^k_\dR(\ol E^k_\Cusp)(\varepsilon) \rightarrow 0.
$$
A standard argument with the Leray sequence and K\"unneth formula for 
de Rham cohomology gives the following.
\begin{lemma}\label{pro: dR sym}
	For $M=\MN_{\bbQ}$, we have isomorphisms
	\begin{align*}
		H^0_\dR(M, \Sym^k \sH(1))   &\isom H^{k}_\dR(E^k)(\varepsilon)  \\
		H^1_\dR(M, \Sym^k \sH(1)) &\isom  H^{k+1}_\dR(E^k)(\varepsilon) 
	\end{align*}
	and 
	\begin{align*}
		H^k_\dR(\ol E^k_\Cusp)(\varepsilon)&\isom H^0_\dR(\Isom)^{(k)}.
	\end{align*}
\end{lemma}

\begin{remark}\label{rem: vanishing}
		 It is known that
		$$
			H^{k}_\dR(E^k)(k)(\varepsilon)  \cong H^0_\dR(M, \Sym^k \sH) = 0.
		$$
		We will prove a version of this statement for rigid cohomology on the ordinary locus in \S \ref{section: 4-3}.
\end{remark}

We denote the resulting residue map
$$
	\res^k:H^1_\dR(M, \Sym^k \sH(1))\to H^0_\dR(\Isom)^{(k)}
$$
again by $\res^k$.  As a consequence, the regulator $r_\dR$ from motivic to de Rham cohomology induces 
in the case where $M=\MN_{\bbQ}$ a commutative diagram
\begin{equation}
	\begin{CD}
		H^{k+1}_\mot(E^k,\bbQ(k+1))(\varepsilon)@>{\res^k}>>   H^0_\mot(\Isom,\bbQ(0))^{(k)}\\
		@Vr_\dR VV @VVr_\dR V \\
		H^1_\dR(M, \Sym^k \sH(1)) @>{\res^k}>>   H^0_\dR(\Isom)^{(k)}. 
	\end{CD}
\end{equation}
Let us recall the definition of the Eisenstein class in de Rham cohomology:

\begin{definition}\label{deRhamEis}
	Let $k\ge 0$ and $\varphi$ be a formal linear combination of non-zero torsion sections with 
	coefficients in $\bbQ$.
	The \emph{de Rham Eisenstein class} $\Eis^{k+2}_\dR(\varphi)\in H^1_\dR(M, \Sym^k \sH(1))$ is
	the image of $\Eis^{k+2}_\mot(\varphi)$ under the regulator map
	$$
		r_\dR:H^{k+1}_\mot(E^k,\bbQ(k+1))(\varepsilon)\to H^1_\dR(M, \Sym^k \sH(1)).
	$$
\end{definition}
In particular, the de Rham Eisenstein class $\Eis^{k+2}_\dR(\varphi)$ lies in the zeroth step of
the Hodge filtration
\begin{equation}\label{eq: EisFil0}
	\Eis^{k+2}_\dR(\varphi)\in F^0H^1_\dR(M,\Sym^k \sH(1))=\Gamma( M,\ul\omega^{\otimes k}\otimes\Omega^1_{M}).
\end{equation}
The formula in \ref{respol} gives:
\begin{corollary}\label{cor: resEisdR}
	Let $M=\MN$, $k\ge 0$ and $\varphi\in \bbQ[(\bbZ/N\bbZ)^2\setminus \{0\}]$. Then 
	the Eisenstein class in de Rham cohomology satisfies
	$$
		\res^k(\Eis^{k+2}_\dR(\varphi))(g) = \frac{-1}{N^{k-1}}\rho^k(\varphi)(g)
	$$
	for any $g \in \GL_2(\bbZ/N)$.
\end{corollary}

%
%
\section{Eisenstein class in syntomic cohomology}
%
%

%
\subsection{Definition of the Eisenstein class}
%

In this section, we define the Eisenstein class in syntomic cohomology.
Suppose $K$ is a finite extension of $\bbQ_p$ with ring of integers $\cO_K$, and let 
$\sV = \Spec\, \cO_K$.   Then for any smooth scheme $X$ over $\sV$, Amnon Besser defined the rigid 
syntomic group
$
	H^m_\syn(X, n)
$
independent of any auxiliary data for $X$ and a syntomic regulator map
$$
	r_\syn: H^m_{\mot}(X, n) \rightarrow H^m_\syn(X, n)
$$
(\cite{Bes1} Theorem 7.5.)  We will use the above regulator map to define the syntomic Eisenstein class.
In what follows, we let $K = \bbQ_p$ and $\sV = \Spec\, \bbZ_p$.  

Let $N$ be an integer $\geq 3$ prime to $p$ and $M = M(N)_{\bbZ_p}$  the extension of $M(N)$ to $\bbZ_p$.  
Furthermore, we let $\pi : E \rightarrow M$ be the universal elliptic curve over $M$, and we denote by 
$E^k$ the $k$-fold fiber product of $E$ over $M$.  

\begin{definition}\label{def: syntomic Eisenstein}
	As in Definition \ref{def:motivicEisenstein}, we define the syntomic Eisenstein class $\Eis^{k+2}_\syn(\varphi)$
	to be the image by the syntomic regulator
	$$
		r_\syn: H^{k+1}_\mot(E^k, \bbQ(k+1))(\epsilon) \rightarrow H^{k+1}_\syn(E^k, k+1)(\epsilon)
	$$
	of the motivic Eisenstein class $\Eis^{k+2}_\mot(\varphi)$.
\end{definition}

By construction (\cite{Bes1} Theorem 7.5), the syntomic regulator map is compatible with the de Rham regulator map.  
Hence we have the following.
\begin{lemma}\label{lem: syn to dR}
	The syntomic Eisenstein class  $\Eis^{k+2}_\mot(\varphi)$ maps to the de Rham Eisenstein class 
	 $\Eis^{k+2}_\dR(\varphi)$ through the boundary map
	\begin{equation}\label{eq: syn to dR}
		H^{k+1}_\syn(E^k, k+1)(\epsilon)  \rightarrow H^{k+1}_\dR(E^k_{\bbQ_p})(\epsilon).
	\end{equation}
\end{lemma}

The purpose of this paper is to explicitly describe the syntomic Eisenstein class $\Eis^{k+2}_\syn(\varphi)$,
restricted to the ordinary locus of $M$, in terms of $p$-adic Eisenstein series.    We first describe 
$H^{k+1}_\syn(E^k, k+1)(\epsilon)$ in terms of rigid syntomic cohomology with coefficients.

The theory of rigid syntomic cohomology with coefficients was developed in \cite{Ba1}, and a review of this 
theory is given in the Appendix of this paper.   When the smooth $\sV$-scheme $X$ is part of a smooth pair 
$\sX = (X, \ol X)$, then the syntomic cohomology of Besser corresponds to rigid syntomic with coefficients in 
Tate objects, and we have an isomorphism
$$
	H^m_{\syn}(X, n)  \cong H^m_\syn(\sX, \bbQ_p(n)),
$$
where the right hand side is rigid syntomic cohomology of $\sX$ with coefficients in the Tate object $\bbQ_p(n)$.

If we let $\ol M$ be a smooth compactification of $M$ over $\bbZ_p$, then $\sM = (M, \ol M)$  is a smooth pair.
We let $\ol E$ be the N\'eron model of $E$ over $\ol M$.  Although $\ol E$ is not smooth over $\ol M$, it is smooth over  
$\Spec\, \bbZ_p$.  Hence $\sE = (E, \ol E)$ is a smooth pair such that the morphism
$$
	\pi : \sE \rightarrow \sM
$$
is proper and smooth.    We let $\ol E^k := \ol E \times_{\ol M} \cdots \times_{\ol M} \ol E$ be the $k$-fold fiber product
of $\ol E$ over $\ol M$.   This variety is not smooth over $\bbZ_p$ when $k>1$. 
We denote by $\wt E^k$ the Kuga-Sato variety, which is a 
canonical desingularization of $\ol E^k$ defined in \cite{Del1} and \cite{Sch}.  Then $\sE^k = (E^k, \wt E^k)$ is a smooth pair.  

\begin{definition}\label{def: H syn}
	We define $\sH$ to be the filtered overconvergent $F$-isocrystal
	$$
		\sH := R^1 \pi_*\bbQ_p(1)
	$$
	on $S(\sM)$, where the higher direct image is defined as in Definition \ref{definition: HDI}.
\end{definition}
Again as in Lemma \ref{pro: dR sym}, standard argument with the Leray sequence and Kunneth formula
for de Rham and rigid cohomology gives the following.

\begin{lemma}\label{pro: rig sym}
	We have isomorphisms
	\begin{align*}
		H^0_\rig(\sM, \Sym^k \sH) &\cong H^{k}_\rig(\sE^k)(k)(\varepsilon),\\
		H^1_\rig(\sM, \Sym^k \sH) &\cong H^{k+1}_\rig(\sE^k)(k)(\varepsilon),
	\end{align*}
	which are compatible with the Frobenius and the Hodge filtration.
\end{lemma}

\begin{remark}
		 Implicit in Lemma \ref{pro: rig sym} is the fact that the canonical map 
		 $$
			 H^m_\dR(\sM, \Sym^k \sH) \xrightarrow\cong H^m_\rig(\sM, \Sym^k \sH)
		$$
		is an isomorphism for any integer $m \geq 0$.  We also have
		$$
			 H^m_\dR(\sM, \Sym^k \sH) \xrightarrow\cong H^m_\dR(M_{\bbQ_p}, \Sym^k \sH),
		$$
		in other words, the de Rham cohomology may be calculated on $M_{\bbQ_p}$.  Henceforth, we 
		will freely use this fact.
\end{remark}

Since $M$ is affine, by Remark \ref{rem: Frobenius lifting}, there exists an overconvergent 
Frobenius $\frobphi_M$, and we may consider rigid syntomic cohomology with coefficients of $\sM$.  
We have the following.

\begin{proposition}\label{pro: canonical isom}
	We have canonical isomorphisms
	\begin{equation*}\begin{split}
			H^{k+1}_\syn(E^k, k+1)(\epsilon) &\xrightarrow\cong H^0_\syn(\sV, H^{k+1}_\rig(\sE^k, \bbQ_p(k+1))(\epsilon)),  \\
			H^{1}_\syn(\sM, \Sym^k \sH(1)) &\xrightarrow\cong H^0_\syn(\sV, H^{1}_\rig(\sM, \Sym^k \sH)(1)).
	\end{split}\end{equation*}
\end{proposition}

\begin{proof}
	By \cite{Bes1} Remark 8.7.3 and the isomorphism of de Rham with rigid cohomology, we have a long exact sequence
	\begin{multline*}
		\cdots \rightarrow  H^{k}_\rig(\sE, \bbQ_p(k+1))(\epsilon) \rightarrow H^{k+1}_\syn(E^k, k+1)(\epsilon) \\
		\rightarrow F^{0} H^{k+1}_\rig(\sE^k, \bbQ_p(k+1))(\epsilon)
		 \xrightarrow{1 - \frobphi^*} H^{k+1}_\rig(\sE^k,\bbQ_p(k+1))(\epsilon) \rightarrow \cdots.  
	\end{multline*}
	The first map is obtained from the fact that 
	\begin{multline*}
		H^0_\syn(\sV, H^{k+1}_\rig(\sE^k, \bbQ_p(k+1))(\epsilon)) \\
		= \ker
		\left(  F^0 H^{k+1}_\rig(\sE^k, \bbQ_p(k+1))(\epsilon) \xrightarrow{1-\frobphi^*}
		H^{k+1}_\rig(\sE^k, \bbQ_p(k+1))(\epsilon)
		 \right),
	\end{multline*}
	and we see from the construction that it is surjective.
	Similarly, we have an exact sequence
	\begin{multline*}
		0 \rightarrow H^1_\syn(\sV, H^{0}_\rig(\sM, \Sym^k \sH)(1)) \rightarrow
		H^{1}_\syn(\sM, \Sym^k \sH(1)) 
		\\\rightarrow H^0_\syn(\sV, H^{1}_\rig(\sM, \Sym^k \sH)(1)) \rightarrow 0,
	\end{multline*}
	and the second map is given by the surjection.  The maps are isomorphisms since
	we have
	$$
		H^{k}_\rig(\sE^k)(k)(\epsilon) = H^0_\rig( \sM, \Sym^k \sH) = 0
	$$
	from Remark \ref{rem: vanishing} and the fact that rigid cohomology is isomorphic to de Rham cohomology
	in our case.
\end{proof}

\begin{definition}
	We define 
	\begin{equation}\label{eq: main isom}
		H^{k+1}_\syn(E^k, k+1)(\epsilon) \xrightarrow\cong H^{1}_\syn(\sM, \Sym^k \sH(1))
	\end{equation}
	to be the isomorphism making the diagram
	\begin{equation*}
		\begin{CD}
			H^{k+1}_\syn(E^k, k+1)(\epsilon) @>{\cong}>> H^0_\syn(\sV, H^{k+1}_\rig(\sE^k, \bbQ_p(k+1))(\epsilon))  \\
			@V{\cong}VV @V{\cong}VV\\
			H^{1}_\syn(\sM, \Sym^k \sH(1)) @>{\cong}>> H^0_\syn(\sV, H^{1}_\rig(\sM, \Sym^k \sH)(1))
		\end{CD}
	\end{equation*}
	commutative, where the horizontal maps are the canonical isomorphisms of Proposition \ref{pro: canonical isom}
	and the right vertical isomorphism is induced from Lemma \ref{pro: rig sym}.  
	It is canonical in a sense that it is the composition of canonical maps.
\end{definition}

\begin{definition}
	We denote again by $\Eis_\syn^{k+2}(\varphi)$ the element
	$$
		\Eis_\syn^{k+2}(\varphi) \in H^1_\syn(\sM, \Sym^k \sH(1))
	$$
	which is defined to be the image of the syntomic Eisenstein class of Definition \ref{def: syntomic Eisenstein}
	with respect to the canonical isomorphism of \eqref{eq: main isom}.
\end{definition}

The syntomic Eisenstein class may be characterized as follows.

\begin{proposition}\label{pro: syn to dR}
	The syntomic Eisenstein class $\Eis_\syn^{k+2}(\varphi)$ is characterized as the unique element in 
	$H^1_\syn(\sM, \Sym^k \sH(1))$ which maps to 
	$$
		\Eis_\dR^{k+2}(\varphi)  \in   H^1_\dR(M_{\bbQ_p}, \Sym^k \sH)
	$$
	through the boundary map
	$$
		H^1_\syn(\sM, \Sym^k \sH(1))   \rightarrow H^1_\dR(M_{\bbQ_p}, \Sym^k \sH).
	$$
\end{proposition}

\begin{proof}
	The fact that the syntomic Eisenstein class maps to the de Rham class follows from Lemma
	\ref{lem: syn to dR} and the fact that the diagram 
	$$
	\begin{CD}
	H^0_\syn(\sV, H^{k+1}_\rig(\sE^k, \bbQ_p(k+1))(\epsilon)) @>>> H^{k+1}_\dR(E_{\bbQ_p}^k)(\epsilon)\\
	@V{\cong}VV  @V{\cong}VV\\	
	 H^0_\syn(\sV, H^1_\rig(\sM, \Sym^k \sH(1))) @>>>  H^1_\dR(M_{\bbQ_p}, \Sym^k \sH)
	\end{CD}
	$$
	is commutative.  Here the first vertical map is induced from the isomorphism of Proposition 
	\ref{pro: rig sym} and the second vertical isomorphism  is given by 
	Lemma \ref{pro: dR sym}.  The boundary map is 
	defined to be the composition
	\begin{multline*}
		H^1_\syn(\sM, \Sym^k \sH(1))  \xrightarrow\cong H^0_\syn(\sV, H^1_\rig(\sM, \Sym^k \sH(1))) \\
		\hookrightarrow  H^1_\dR(\sM, \Sym^k \sH) \xrightarrow\cong  H^1_\dR(M_{\bbQ_p}, \Sym^k \sH),
	\end{multline*}
	hence it is injective.  This proves the uniqueness of our class.
\end{proof}

In this paper, we will mainly be interested in the restriction of the Eisenstein class to the ordinary part of the modular curve.   
We denote by $M^\ord$ the open subscheme of $M = M(N)_{\bbZ_p}$ obtained by removing the zero of the Eisenstein series 
$$
	E_{p-1}  \in \Gamma( M, \ul\omega^{\otimes (p-1)})
$$
of weight $p-1$.
We let $\sM^\ord$ be the smooth pair $\sM^\ord = (M^\ord, \ol M)$.  The overconvergent Frobenius $\frobphi_\sM$ on $\sM$
induces an overconvergent Frobenius for $\sM^\ord$, and the inclusion $\sM^\ord \hookrightarrow \sM$ is compatible
with the action of this Frobenius.  We have a pullback map for rigid syntomic cohomology
$$
	H^1_\syn(\sM, \Sym^k \sH(1)) \rightarrow H^1_\syn(\sM^\ord, \Sym^k \sH(1)),
$$
and we denote again by $\Eis^{k+2}_\syn(\varphi)$ the pull back
$$
	\Eis^{k+2}_\syn(\varphi) \in  H^1_\syn(\sM^\ord, \Sym^k \sH(1))
$$
of $\Eis^{k+2}_\syn(\varphi)$ by this map.
We will explicitly describe this cohomology class in terms of $p$-adic Eisenstein series.

%
\subsection{Characterization of the syntomic Eisenstein class}
%

In this section, we prove that unlike the de Rham case, the syntomic Eisenstein class is uniquely characterized by its residue. 
The result of this section will not be used in the proof of our main theorem.   We first define the residue morphism for rigid 
cohomology (with filtration) by pasting together de Rham and rigid cohomology.   We let 
$E^k_\Cusp = \wt E^k \setminus E^k$, which is smooth over $\bbZ_p$, and consider the pairs $\ol\sE^k = (\wt E^k, \wt E^k)$ and  
$\ol\sE^k_\Cusp = (E^k_\Cusp, E^k_\Cusp)$.  Then we have morphisms of smooth pairs
$$
	\xymatrix{
		\sE^k  \ar@{^{(}->}[r]<-0.5ex> &  \ol\sE^k  & \ar@{_{(}->}[l] <0.5ex> \ol\sE^k_\Cusp.
	}
$$
By taking the Gysin exact sequence of rigid cohomology with trivial coefficients and then taking 
the projector $\varepsilon$, we have an exact sequence
{\small \begin{equation}\label{eq: rigid gysin}
	0 \rightarrow H_\rig^{k+1}( \ol \sE^k)(k+1)(\varepsilon) \\
	\rightarrow H^{k+1}_\rig(\sE^k)(k+1)(\varepsilon) 
	\xrightarrow{\res} H^0_\rig(\ol\sE^k_\Cusp)(\epsilon) \rightarrow 0.
\end{equation}}
We have a canonical isomorphism
$$
	 H^0_\rig(\ol\sE^k_\Cusp)(\epsilon) \xrightarrow\cong H^0_\rig(\Isom)^{(k)}.
$$
As in the de Rham case, we define the residue morphism $\res^k$ to be the map making the following diagram commutative.
$$
		\begin{CD}
			H^{k+1}_\rig(\sE^k)(k+1)(\varepsilon) @>>>  H^{0}_\rig(\ol\sE^k_\Cusp)(\varepsilon) \\
			@A{\cong}AA @A{\cong}AA \\
			H^1_\rig(\sM, \Sym^k \sH)(1) @>{\res^k}>>   H^0_\rig(\Isom)^{(k)}.
		\end{CD}
$$	
It is known that the action of the Frobenius on  $H_\rig^{k+1}( \ol \sE^k)(k+1)$ is of pure weight $-k-1$, hence we have
$$
	H^0_\syn(\sV, H_\rig^{k+1}( \ol \sE^k)(k+1)(\epsilon)) = 0.
$$
Using the fact that $ H^0_\syn(\Isom)^{(k)} = H^0_\syn(\sV,   H^0_\rig(\Isom)^{(k)})$,
the above result and  \eqref{eq: rigid gysin} shows that the residue morphism gives an isomorphism
\begin{equation}\label{eq: comp one}
	H^0_\syn(\sV, H_\rig^1(\sM, \Sym^k \sH)(1)  ) \xrightarrow\cong  H^0_\syn(\Isom)^{(k)}.
\end{equation}

\begin{definition}
	We define the residue map for syntomic cohomology
	\begin{equation}\label{eq: syn res}
		\res^k_\syn: H^1_\syn(\sM, \Sym^k \sH(1)) \xrightarrow\cong H^0_\syn(\Isom)^{(k)}
	\end{equation}
	to be the isomorphism obtained as the composition of
	$$
		H^1_\syn(\sM, \Sym^k \sH(1)) \xrightarrow\cong H^0_\syn(\sV, H^1_\rig(\sM, \Sym^k \sH(1)))
	$$
	with \eqref{eq: comp one}.
\end{definition}

Using this map, we may now characterize the syntomic Eisenstein class by its residue.

\begin{proposition}\label{pro: syn res}
	The syntomic Eisenstein class
	$$
		\Eis_\syn^{k+1}(\varphi) \in H^1_\syn(\sM, \Sym^k \sH(1))
	$$
	is characterized as the unique element which satisfies
	$$
		\res^k_\syn(\Eis_\syn^{k+1}(\varphi) ) = \frac{-1}{N^{k-1}} \rho^k(\varphi)(g),
	$$
	where $\res^k_\syn$ is the syntomic residue morphism \eqref{eq: syn res}.
\end{proposition}

\begin{proof}
	We have a commutative diagram
	$$
		\begin{CD}
			H^0_\syn(\sV, H^1_\rig(\sM, \Sym^k \sH(1))) @>{\cong}>> H^0_\syn(\Isom)^{(k)} \\
			@VVV  @VVV \\
			H^1_\dR(M_{\bbQ_p}, \Sym^k \sH) @>{\res^k}>> H^0_\dR(\Isom)^{(k)},
		\end{CD}
	$$
	where the vertical maps are the natural injection.  The calculation of residue follows from
	Proposition \ref{pro: syn to dR}, which asserts that the image of $\Eis_\syn^{k+2}(\varphi)$ 
	in $H^1_\dR(M_{\bbQ_p}, \Sym^k \sH)$ is equal to the de Rham Eisenstein class
	$\Eis_\dR^{k+2}(\varphi)$, and the calculation in Corollary \ref{cor: resEisdR} of the de Rham Eisenstein class.
	The uniqueness follows since $\res^k_\syn$ is an isomorphism.
\end{proof}

%
%
%
\section{Explicit formulas}
%
%
%

In this section we relate the Eisenstein classes of section \ref{deRhamEis} to the Eisenstein series
considered by Katz in \cite{Ka3}. We use the comparison theorem of Beilinson-Levin \ref{respol}.

%
\subsection{Modular forms}
%

Let us assume that $N\ge 3$, so that the $\GarithN$- and the $\Gamma(N)$-moduli problems
are representable. Recall that we defined in (\ref{ulomegadefn}) the coherent sheaf
$$
	\ul\omega:=\pi_*\Omega^1_{E/M}
$$
on $M=\MN$ or $M=\MarithN$. 
\begin{definition}
	Let $k\in \bbZ$. A \emph{modular form $F$ of weight $k+2$} on $M=\MN$ or $M=\MarithN$ is a global
	section
	$$
		F\in \Gamma(M,\ul\omega^{\otimes k+2}).
	$$
\end{definition}
Using the \emph{Kodaira-Spencer isomorphism} 
$$
	\ul\omega^{\otimes 2}\isom \Omega^1_M
$$
we can identify the space of modular forms of weight $k+2$
$$
	\Gamma(M,\ul\omega^{\otimes k+2})\isom \Gamma(M,\ul\omega^{\otimes k}\otimes\Omega^1_M).
$$
In particular, using 
$$
	\Gamma(M,\ul\omega^{\otimes k}\otimes\Omega^1_M)=F^0H^1_\dR(M, \Sym^k \sH(1))
$$ 
we may consider the modular forms of weight $k+2$ as elements in 
$$
	H^1_\dR(M, \Sym^k \sH(1)).
$$

Finally, we note that the right action of $\GL_2(\bbZ/N\bbZ)$ on $\MN$ induces a left action
of $g\in \GL_2(\bbZ/N\bbZ)$ on $F\in \Gamma(\MN,\ul\omega^{\otimes k}\otimes\Omega^1_M)$ by pull-back:
$$
	F\mapsto g^*F.
$$

%
\subsection{The $q$-expansion principle}
%

Consider the Tate curve 
$$
	(\Tate(q^N),\omega_\can,\beta_\can)
$$ 
with its canonical invariant differential $\omega_\can$ and its canonical $\GarithN$-level structure 
$\beta_\can$ over $\bbZ((q))$ (see \cite{Ka3} 2.2). For $N\ge 3$, we get a map
$$
	\iota_\infty:\bbZ((q))\to \MarithN.
$$
The differential $\omega_\can$ provides a basis
of $\iota_\infty^*\ul\omega$ over $\bbZ((q))$ and hence an identification
$$
	\Gamma(\Spec\,\bbZ((q)),\iota_\infty^*\ul\omega^{\otimes k})\isom \bbZ((q)).
$$
Using the compatibility with base change we also get 
$$
	\Gamma(\Spec\,(\bbZ[1/N,\zeta_N]\otimes_{\bbZ}\bbZ((q))),\iota_\infty^*\ul\omega^{\otimes k})\isom \bbZ[1/N,\zeta_N]\otimes_{\bbZ}\bbZ((q)).
$$
Note that $\bbZ[1/N,\zeta_N]\otimes_{\bbZ}\bbZ((q))\subset\bbZ[1/N,\zeta_N]((q))$
\begin{definition}\label{q-expansiondefn}
	The \emph{$q$-expansion homomorphism} is the map
	\begin{align*}	
		q_\infty:\Gamma(\MarithN,\ul\omega^{\otimes k+2})&\to \bbZ((q))\\
		F&\mapsto \iota_\infty^*F.
	\end{align*}
	In the same way we get a $q$-expansion map
	\begin{align*}	
		q_\infty:\Gamma(\MN,\ul\omega^{\otimes k+2})&\to \bbZ[1/N,\zeta_N]((q))\\
		F&\mapsto \iota_\infty^*F.
	\end{align*}
\end{definition}
Note that the base change map $\Gamma(\MarithN,\ul\omega^{\otimes k+2})\to \Gamma(\MN,\ul\omega^{\otimes k+2})$
is injective and that we have a commutative diagram
\begin{equation}\label{q-exp-CD}
	\begin{CD}
		\Gamma(\MarithN,\ul\omega^{\otimes k+2})@>q_\infty>>\bbZ((q))\\
		@VVV@VVV\\
		\Gamma(\MN,\ul\omega^{\otimes k+2})@>q_\infty>> \bbZ[1/N,\zeta_N]((q)).
	\end{CD}
\end{equation}
	
\begin{theorem}[$q$-expansion principle]\label{q-exp-thm}
For a fixed weight $k+2\ge 0$ the homomorphisms $q_\infty$ are injective:
$$
	q_\infty:\Gamma(\MarithN,\ul\omega^{\otimes k+2})\hookrightarrow \bbZ((q))
$$
and 
$$
	q_\infty:\Gamma(\MN,\ul\omega^{\otimes k+2})\hookrightarrow \bbZ[1/N,\zeta_N]((q)).
$$
\end{theorem}
For a proof see \cite{Ka1} Corollary 1.6.2.

%
\subsection{Eisenstein series}
%

Let
$$
	\bbQ[(\bbZ/N\bbZ)^2]:=\{\varphi:(\bbZ/N\bbZ)^2\to \bbQ\}
$$
be the space of $\bbQ$-valued functions on $(\bbZ/N\bbZ)^2$.
We want to define explicitly Eisenstein series on $\MN$. For this 
we consider $M(N)$ over $\Spec\,\bbQ$ and 
writing down Eisenstein series on $\MN(\bbC)$. Then we use the
$q$-expansion principle to show that these Eisenstein series are in fact 
defined over $\bbQ$ and that they are in fact already modular forms on $\MarithN$. 
Note that we can parametrize as in \cite{HK2} section 7
$$
	\MN(\bbC)=\SL_2(\bbZ)\backslash(\frH\times \GL_2(\bbZ/N)).
$$
A neighborhood around the cusps is then given by 
$$
	\pm U(\bbZ)\backslash (\frH\times \GL_2(\bbZ/N))\isom \bbC^\times \times \left(\pm U(\bbZ/N)\backslash\GL_2(\bbZ/N)\right),
$$
where $\pm U:= \left\{\pm \begin{pmatrix} 1&*\\0&1\end{pmatrix} \right\} $.
Here the map is given by $(\tau,g)\mapsto (e^{2\pi i\tau/N},g)$. The cusp $\infty$
corresponds in this description to the coset of $\id$.

We are going to define certain holomorphic Eisenstein series $E_{k+2,0,\varphi}$. These will be used to 
express $\Eis^{k+2}_\dR(\varphi)$ explicitly and are essentially the Eisenstein series used by
Katz to define his $p$-adic measure. 
\begin{definition}
	Let $k\ge 1$ and $\varphi\in \bbQ[(\bbZ/N\bbZ)^2]$.
	Define a holomorphic \emph{Eisenstein series} on $\MN(\bbC)$ by the formula
	$$
		E_{k+2,0,\varphi}(\tau,g):=\frac{(-1)^{k+2}N^{k+2}(k+1)!}{2(2\pi i)^{k+2}}\sum_{(m,n)\in\bbZ^2\setminus (0,0)}
		\frac{\wh{g\varphi}(m,n)}{(m+n\tau)^{k+2}},
	$$
	where $\tau$ is the coordinate in the upper half plane and $\wh{g\varphi}$ the symplectic Fourier transform
	 introduced in \eqref{symplecticFourier}.
\end{definition}

\begin{lemma}\label{lem: comparison Katz}
	Consider the Eisenstein series $G_{k+2,0,f}$ on $\MarithN$ for $f\in \bbQ[(\bbZ/N\bbZ)^2]$
	defined in Katz \cite{Ka3} 3.6.9. 	
	If one identifies $\MarithN(\bbC)$ with the component of $g\in \GL_2(\bbZ/N)$ in $\MN(\bbC)$, one gets
	$$
		E_{k+2,0,\varphi}(\tau,g)=G_{k+2,0,P_1(\wh{g\varphi})}(\tau),
	$$
	where $P_1(\wh{g\varphi})$ is the composition of the Fourier transforms defined in \eqref{partialFourier} and \eqref{symplecticFourier}
	for $g \varphi$.
\end{lemma}
\begin{proof}
	This follows directly from the definition.
\end{proof}
To define the Eisenstein measure later, we need to know the $q$-expansion of $E_{k+2,0,\varphi}$.
\begin{lemma}
	Let $k\ge 1$ and $\varphi\in \bbQ[(\bbZ/N\bbZ)^2]$.
	Then the $q$-expansion of $E_{k+2,0,\varphi}$ at the cusp $g\in \GL_2(\bbZ/N)$ is given by
	\begin{multline}\label{eq: q expansion}
		\frac{1}{2}L \left(-1-k, P_1(g\varphi)(0,m) -(-1)^{k+1} P_1(g\varphi)(0,-m)\right)\\+\sum_{n>0}
		q^n \sum_{n=dd'}  \left( d^{k+1}P_1(g\varphi)(d',d)- (-d)^{k+1} P_1(g \varphi) (-d', -d) \right).
	\end{multline}
\end{lemma}
\begin{proof}
	For the computation at the cusp $\id=\infty$, see for example \cite{Ka3} 3.2.5.  For the general case,
	we use that $E_{k+2,0,\varphi}(\tau,g)=E_{k+2,0,g\varphi}(\tau,\id)$.
\end{proof}
If $\varphi\in\bbQ[(\bbZ/N\bbZ)^2]$ this $q$-expansion has coefficients in $\bbQ$
and the $q$-expansion principle, (\ref{q-exp-CD}) and \ref{q-exp-thm}, allow us to conclude:
\begin{proposition}[$q$-expansion of Eisenstein series \cite{Ka3} 3.6.9.]\label{Eis-q-expansion}
	Let $k\ge 1$ and $\varphi\in \bbQ[(\bbZ/N\bbZ)^2]$. Then there are modular forms of weight $k+2$
	$$
		E_{k+2,0,\varphi}\in \Gamma(\MN_\bbQ,\ul\omega^{\otimes k+2})
	$$
	called \emph{Eisenstein series of weight $k+2$}, whose $q$-expansion on the component
	$g\in\GL_2(\bbZ/N) $ is given by \eqref{eq: q expansion}.
\end{proposition}

%
\subsection{Determination of the de Rham Eisenstein class}
%

We will determine in this paragraph the exact relation between the
the Eisenstein series $E_{k+2,0,\varphi}$ and the de Rham Eisenstein class
$\Eis^{k+2}_\dR(\varphi)$.

Consider the Eisenstein series $E_{k+2,0,\varphi}$ and form the section 
	$$
		E_{k+2,0,\varphi}\frac{dq}{q}\wedge dz_1\wedge\ldots\wedge dz_k
	$$
of $\Gamma(\MN_\bbQ,\Omega^1\otimes\ul\omega^{\otimes k})$. Using the
Kodaira-Spencer isomorphism $\Omega^1\isom\ul\omega^{\otimes 2}$ one can consider this also as
a section of $\ul\omega^{\otimes k+2}$.
\begin{proposition} 
	The Eisenstein class in de Rham cohomology $\Eis^{k+2}_\dR(\varphi)$ is given by
	\begin{equation}\label{eq: de Rham Eisenstein}
		\Eis^{k+2}_\dR(\varphi)=\frac{2}{N^{k+1}k!}E_{k+2,0,\varphi}\frac{dq}{q}\wedge dz_1\wedge\ldots\wedge dz_k.
	\end{equation}
\end{proposition}
\begin{proof}
	Using \cite{Bei} 2.1.3. and the explicit formula in \cite{HK2} p. 329 top, the 
	section 
	$$
		\frac{(2\pi i)^k}{N}\sum_{\gamma\in \pm U(\bbZ)\backslash \Sl_2(\bbZ)}\frac{\rho(\varphi)(\gamma g)}{(c\tau +d)^{k+2}}
		\frac{dq}{q}\wedge dz_1\wedge\ldots\wedge dz_k
	$$
	of $\Gamma(\MN_\bbQ,\Omega^1\otimes\ul\omega^{\otimes k})$ has residue $\rho(\varphi)$ and represents
	the de Rham realization of the Eisenstein symbol. On the other hand
	$\Eis^{k+2}_\dR(\varphi)$ is a multiple of the Eisenstein symbol and has residue $\frac{-1}{N^{k-1}}\rho(\varphi)$ 
	following \ref{cor: resEisdR}. Thus, $-N^{k-1}\Eis^{k+2}_\dR(\varphi)$ is the Eisenstein symbol. 
	Using the definition of $\rho(\varphi)$ one computes (see \cite{HK2} p. 334 bottom) that the Eisenstein symbol is in fact
	$$
		\frac{(-1)^{k+1}N^k(k+1)}{(2\pi i)^{k+2}}\sum_{(m,n)\in\bbZ^2\setminus(0,0)}
		\frac{\widehat{(g\varphi)}(m,n)}{(m+n\tau)^{k+2}},
	$$
	which is $\frac{-2}{N^{2}k!}E_{k+2,0,\varphi}$ by inspection. Putting everything together gives the desired result. 
\end{proof}

%
%
%
\section{Syntomic class in the ordinary locus}
%
%
%

In this section, we give a characterization of the restriction of the syntomic Eisenstein class
to the ordinary locus.

%
\subsection{Characterization of the ordinary class}
%

We denote again by $M^\ord$ the open subscheme of $M = M(N)_{\bbZ_p}$ 
obtained by removing the zero of the Eisenstein series 
$
	E_{p-1}  \in \Gamma( M, \ul\omega^{\otimes (p-1)}),
$
and we let $\sM^\ord = (M^\ord, \ol M)$.  The main result of this paper is the explicit description of the syntomic Eisenstein
class on the ordinary locus
$$
	\Eis^{k+2}_\syn(\varphi) \in  H^1_\syn(\sM^\ord, \Sym^k \sH(1)).
$$
It is again characterized as in Proposition \ref{pro: syn to dR} by the de Rham class as follows.

\begin{proposition}\label{pro: ord to dR}
	The syntomic Eisenstein class $\Eis^{k+2}_\syn(\varphi)$ restricted to the ordinary locus 
	is the unique class in $ H^1_\syn(\sM^\ord, \Sym^k \sH(1))$
	which maps to the de Rham class 
	$$
		\Eis^{k+2}_\dR(\varphi) \in H^1_\dR(\sM^\ord, \Sym^k \sH)
	$$ 
	with respect to the boundary map
	$$
		H^1_{\syn}(\sM^\ord, \Sym^k \sH(1)) \rightarrow H^1_\dR(\sM^\ord, \Sym^k \sH).
	$$
\end{proposition}

\begin{proof}
	The fact that the syntomic class maps to the de Rham class follows from Proposition \ref{pro: syn to dR}
	and the fact that the following diagram
	$$
		\begin{CD}
			H^1_{\syn}(\sM, \Sym^k \sH(1)) @>>> H^1_\dR(M_{\bbQ_p}, \Sym^k \sH)\\
				@VVV @VVV\\
			H^1_{\syn}(\sM^\ord, \Sym^k \sH(1)) @>>> H^1_\dR(\sM^\ord, \Sym^k \sH)
		\end{CD}
	$$
	 is commutative, where the vertical maps are the pullbacks.  The boundary map is define as the composition
	 of the surjection in the short exact sequence
	 \begin{multline*}
	 	0 \rightarrow H^1_\syn( \sV, H^0_\rig(\sM^\ord, \Sym^k \sH(1))) 
		\rightarrow H^1_\syn(\sM^\ord, \Sym^k \sH(1)) \\
		\rightarrow H^0_\syn(\sV, H^1_\rig(\sM^\ord, \Sym^k \sH(1)))
		\rightarrow 0
	 \end{multline*}
	 and the injection 
	 $$
	 	H^0_\syn(\sV, H^1_\rig(\sM^\ord, \Sym^k \sH(1))) \hookrightarrow H^1_\dR(\sM^\ord, \Sym^k\sH).
	$$
	By Lemma \ref{lem: ord vanish} below, the boundary map in injective.
	Thus we have the uniqueness.
\end{proof}

What now remains for the proof of Proposition \ref{pro: ord to dR} is the following lemma.

\begin{lemma}\label{lem: ord vanish}
	We have $H^0_\rig(\sM^\ord, \Sym^k \sH(1)) = 0$.
\end{lemma}

The proof of this lemma will be given in  \S \ref{section: 4-3}. 
The advantage of considering $\sM^\ord$ over $\sM$ is that it is equipped with an overconvergent Frobenius
which has a natural interpretation in terms of the moduli problem.  We now proceed to describe this Frobenius.
First, let $M^\ord_\arith$ be the open subscheme of $M_\arith := M_\arith(N)$ obtained by removing the zero of $E_{p-1} 
\in \Gamma(M_\arith, \ul\omega^{\otimes (p-1)})$, and denote by $\cM_\arith^\ord$ the formal completion of 
$M_\arith^\ord$ with respect to the special fiber.  This scheme parameterizes isomorphism classes
of $\Gamma(N)^\arith$ elliptic curves $(E, \beta)$ over $\bbZ_p$ such that $E$ is \textit{ordinary}, in other 
words, the Hasse invariant of $E$ is one.  The lifting of the kernel of the absolute Frobenius of the special 
fiber gives a subgroup $H$ of $E$ of order $p$.  The correspondence
$$
	(E, \beta) \mapsto (E/H, \beta'),
$$
where $\beta'$ is defined as the composition $\mu_N \times \bbZ/N \xrightarrow\beta E \rightarrow E/H$
defines a morphism of moduli spaces
\begin{equation}\label{eq: Frob first}
	\Frob: \cM^\ord_\arith \rightarrow \cM^\ord_\arith
\end{equation}
over $\bbZ_p$.  Denote by $\cM^\ord_{\arith \bbQ_p}$ the rigid analytic space over $\bbQ_p$
associated to the formal scheme $\cM^\ord_\arith$.
 By \cite{Ka1} Theorem 3.1, the construction of $H$ is known to extend to a certain
strict neighborhood $U$ of $\cM^\ord_{\arith  \bbQ_p}$ in $\ol M^\an_{\bbQ_p}$.
Hence $\Frob$ also extends to this strict neighborhood, implying that $\Frob$ is in fact 
overconvergent in the sense of Definition \ref{def: overconvergent Frobenius}.

Let $\cM^\ord$ be the formal completion of $M^\ord$ with respect to the special fiber.  Then
we define $\phi_{\cM^\ord} : \cM^\ord \rightarrow \cM^\ord$ to be the Frobenius on $\cM^\ord$ 
defined as the tensor product 
\begin{equation}\label{eq: phi first}
	\phi_{\cM^\ord} := \Frob \otimes \sigma
\end{equation}
through the isomorphism
$
	\cM^\ord = \cM^\ord_\arith \otimes_{\bbZ_p} \bbZ_p[\zeta_N],
$
where $\Frob$ is as above and $\sigma$ is the canonical Frobenius on $\bbZ_p[\zeta_N]$
lifting the absolute Frobenius of the special fiber.  This gives a lifting of the absolute Frobenius of the
special fiber of $\cM^\ord$, and  since $\Frob$ is overconvergent, $\phi_{\cM^\ord}$
is also overconvergent.

%
\subsection{Moduli space of trivialized elliptic curves}
%

The difficulty in explicitly describing the cohomology of $\Sym^k \sH$ stems from the fact that $\sH$ 
is only locally free and does not have a free basis over $\cM^\ord_{\bbQ_p}$.  We introduce here a certain
$p$-adic universal covering over $\cM^\ord_{\bbQ_p}$ such that the pull-back of $\sH$ to this covering
is free.

Suppose $B$ is a $p$-adic ring, i.e., a ring which is complete 
and separated in the $p$-adic topology. A \emph{trivialized elliptic curve}
$(E/B,\eta)$ is a pair consisting of an elliptic curve $E/B$ together with
an isomorphism of formal groups 
$$
	\eta:\wh E\isom \wh\bbG_m
$$
over $B$.  Let $N \geq 3$ be an integer prime to $p$.   We let 
$$
	\wt \cM_\arith := M(\Gamma_{00}(p^\infty) \cap \Gamma(N)^\arith)_{\bbZ_p},
$$
which parameterizes the isomorphism class of trivialized elliptic curves over $\bbZ_p$
with $\GarithN$-structure.  This is an affine scheme, and since any trivialized $E/B$ is
ordinary, $\wt \cM_\arith$ is a covering of $\cM_\arith$.  We let
$$
	V(\bbZ_p,\GarithN) := \Gamma( \wt \cM_\arith, \cO_{\wt \cM_\arith}).
$$
An element $F\in V(\bbZ_p,\GarithN)$ is called a \emph{generalized $p$-adic modular form}.
For any $p$-adic ring $B$ the above functor restricted to $B$ is represented by
$$
	V(B,\GarithN):=V(\bbZ_p,\GarithN)\otimes_{\bbZ_p} B.
$$
We let $G(N):= \bbZ_p^\times \times (\bbZ/N\bbZ)^\times$. Then the group $G(N)$ acts on 
$F\in V(B,\GarithN)$ by the formula
$$
	[a,b]F(E,\eta,\beta):=F(E,a^{-1}\eta,\beta\verk (b,b^{-1})).
$$
Let $\chi:G(N)\to B^\times$ be a continuous character.  We say that an element $F \in V(B,\GarithN)$ is
\emph{of weight $\chi$}, if
$$
	[a,b]F=\chi(a,b)F
$$
for all $(a,b)\in G(N)$. If $\chi$ is of the form $\chi_k\rho$, where $\rho$ is a character of
finite order on $G(N)$ and $\chi_k(a,b)=a^k$, then one calls $F$ of weight $k$ and Nebentypus $\rho$.

A trivialization $\eta:\wh E\isom \wh\bbG_m$ induces a differential $\omega_\eta$ on $E$ by pulling
back the standard differential $dT/(1+T)$ on $\wh\bbG_m$. This differential $\eta^*(dT/(1+T))$ is 
necessarily the restriction of a differential $\omega_\eta$ on $E$. If $B/\bbZ_p$ is flat, then
conversely the trivialization $\eta$ is uniquely determined by $\omega_\eta$. For this let 
$u$ be a formal parameter for $\wh E$ and integrate $\omega_\eta$ formally over $B\otimes \bbQ$, i.e.,
writing $\omega_\eta=d\Psi(u)$ with $\Psi(u)=\sum_{n\ge 1}a_nu^n $ with $a_n\in B\otimes \bbQ$.
Then $u\mapsto \exp(\Psi(u))$ gives the trivialization $\eta$. This construction 
$$
	(E,\eta,\beta)\mapsto (E,\eta^*(dT/(1+T)),\beta)
$$
induces a homomorphism
\begin{align}
	\Gamma(\MarithN,\shomega^{\otimes k+2})&\to V(\bbZ_p,\GarithN)\\
\nonumber  F&\mapsto \wt F
\end{align}
where $\wt F(E,\eta,\beta):=F(E,\eta^*(dT/(1+T)),\beta)$.  Thus a modular form in the
usual sense naturally gives a  generalized $p$-adic modular form.   We 
define a graded subring $GV^{\cdot} (\bbZ_p,\GarithN)\subset V(\bbZ_p,\GarithN)$.
We say that $F\in GV^k(\bbZ_p,\GarithN)$ iff for all $a\in \bbZ_p^\times$, we have
$$
	[a,1]F=a^kF.
$$

Finally note that we have a $q$-expansion principle. The Tate curve 
$$
	(\Tate(q^N),\omega_\can,\beta_\can)
$$
viewed over $\wh{\bbZ_p((q))}$, the $p$-adic completion of $\bbZ_p((q))$, has a canonical 
trivialization by noting that its formal group is by construction $\wh\bbG_m$. Evaluation at
$\Tate(q^N)$ then defines an injective $q$-expansion map
$$
	q_\infty:V(B,\GarithN)\hookrightarrow \wh{B((q))}.
$$
by construction this $q$-expansion is compatible with the $q$-expansion for 
$\Gamma(\MarithN,\shomega^{\otimes k+2})$.

Similarly, we let
$$
	\wt \cM = M(\Gamma_{00}(p^\infty) \cap \Gamma(N))_{\bbZ_p},
$$
which parameterizes the isomorphism class of trivialized elliptic curves over $\bbZ_p$
with a $\Gamma(N)$-structure.   Note that we have an isomorphism
\begin{equation}\label{eq: parts}
	\wt \cM \cong \wt \cM_\arith \otimes_{\bbZ_p} \bbZ_p[\zeta_N].
\end{equation}
We let
$
	V(\bbZ_p, \Gamma(N)) := \Gamma(\wt \cM, \cO_{\wt \cM}).
$
Then the isomorphism \eqref{eq: parts} implies that
$$
	V(\bbZ_p, \Gamma(N)) = V(\bbZ_p, \GarithN) \otimes_{\bbZ_p} \bbZ_p[\zeta_N].
$$

%
\subsection{The Frobenius and the Gauss-Manin connection}
%

We first describe the Frobenius $\Frob$ and $\phi_{\wt\cM}$ on $\wt\cM_\arith$ and $\wt\cM$
lifting the Frobenius $\Frob$ and $\phi_{\cM^\ord}$ on $\cM^\ord_\arith$ and $\cM^\ord$.  Then
we discuss the Frobenius and the Gauss-Manin connection on $\sH$.

Let $(E,\eta,\beta)$ be a trivialized $\GarithN$-elliptic curve. Define
$$
	E':=E/\eta^{-1}(\mu_p)
$$
and let $\pi:E\to E'$ be the canonical map. Then $\pi^t:E'\to E$ is \'etale
and we define $\eta':=\eta\verk\pi^t$. As usual a $\GarithN$-structure $\beta$ on
$E$ gives rise to a $\GarithN$-structure $\beta'$ on $E'$ (see \cite{Ka3} 5.5.0.)
We define the Frobenius endomorphism
$$
	\Frob: \wt \cM_\arith \rightarrow \wt \cM_\arith
$$
to be the morphism induced from $(E, \eta, \beta) \mapsto (E', \eta', \beta')$.
This morphism naturally lifts the Frobenius morphism \eqref{eq: Frob first},
and induces the morphism
$
	\Frob :  V(B,\GarithN)\to V(B,\GarithN)
$
on the global section of  $\wt \cM$ given by
$$
	\Frob \, F(E,\eta,\beta):=F(E',\eta',\beta').
$$
As $\Frob(\Tate(q^N),\omega_\can,\beta_\can)=(\Tate(q^{pN}),\omega_\can,\beta_\can)$
the effect on the $q$-expansion is $\Frob F(q)=F(q^p)$. Note finally (\cite{Ka3} 5.5.8.) 
that $\Frob $ commutes with the action of $G(N)$.  For the case of full level $N$-structure, the Frobenius 
\begin{equation}\label{eq: Frob full level}
	\frobphi_{\wt\cM}: \wt \cM \rightarrow \wt \cM
\end{equation}
on
$
	\wt \cM \cong \wt \cM_\arith \otimes_{\bbZ_p} \bbZ_p[\zeta_N].
$
is given as the tensor product $\frobphi_{\wt\cM} := \Frob \otimes \sigma$, where $\Frob$ is as above and $\sigma$
is the Frobenius on $\bbZ_p[\zeta_N]$.

For each $N\ge 3$ one can define a derivation $N\theta:V(\bbZ_p,\GarithN)\to V(\bbZ_p,\GarithN)$
by using the square of the canonical form $\eta^*(dT/(1+T))$ and the Kodaira-Spencer isomorphism
$\shomega^{\otimes 2}\isom \Omega^1_{\wt\cM_\arith/\bbZ_p}$ to define a global section of 
$\Omega^1_{\wt\cM_\arith/\bbZ_p}$.  The derivation $N\theta$ is then the dual of this global section. We recall from
\cite{Ka3} 5.8.1. the main property of $N\theta$. The following diagram commutes
\begin{equation}\label{eq: N theta}
	\begin{CD}
		V(\bbZ_p,\GarithN)@>N\theta >>V(\bbZ_p,\GarithN)\\
		@Vq_\infty VV@V q_\infty VV\\
		\wh{\bbZ_p((q))}@>q\frac{d}{dq} >>\wh{\bbZ_p((q))}.
	\end{CD}
\end{equation}
Moreover, for $(a,b)\in G(N)$ one has
$$
	[a,b]\verk N\theta=a^2N\theta\verk[a,b].
$$
The same derivation is defined also for $V(\bbZ_p, \Gamma(N))$.

We now consider the filtered overconvergent Frobenius isocrystal $\sH$ on $\cM^\ord_{\bbQ_p}$.
Let $\sH := R^1 \pi_* \bbQ_p(1)$ as in Definition \ref{def: H syn}, and we denote by $\wt \sH$ the
pull back of $\sH$ to $\wt \cM_{\bbQ_p}$.  
We now explicitly calculate the Frobenius and the Gauss-Manin connection on $\wt\sH$.

We denote by $\wt E$ the universal elliptic curve over $\wt \cM$.  Then this curves
has a universal trivialization
$$
		\eta: \wh\bbG_m  \cong \wt E
$$
over $\wt \cM$ which gives rise to a canonical section $\wt\omega$ of $\ul\omega$, characterized
by the property that $\wt\omega$ restricts to $\eta^*(dT/(1+T))$ on $\wh\bbG_m$.
On the Tate module, $\wt\omega$
coincides with the canonical differential $\omega_\can$ in the usual sense.
Since the scheme $\wt \cM$ is affine,  we may take sections $x$ and $y$ on 
$\wt E$ such that $\wt E$ is defined by the Weierstrass equation
$$
	\wt E: y^2 = 4 x^3 - g_2 x - g_3, \qquad g_2, g_3 \in V(\bbZ_p, \GarithN)
$$
and $\wt\omega = dx/y$. We let $\wt\eta := x dx/y$.  Then $\{ \wt\omega, \wt\eta\}$ 
form a basis of $\wt\sH^\vee$.  This choice gives a splitting
$$
	\wt\sH^\vee \cong \ul\omega \bigoplus \ul\omega^{-1}.
$$
By \cite{Ka1} Lemma (A2.1), the Frobenius on this module acts  as
$$
	\Phi \begin{pmatrix} \wt\omega \\   \wt\eta \end{pmatrix}  =
	\begin{pmatrix}   p/\lambda  & 0  \\  c   &  \lambda  \end{pmatrix}
	\begin{pmatrix} \wt\omega \\   \wt\eta \end{pmatrix}
$$
for some $\lambda$ invertible in $V(\bbZ_p,\Gamma(N))$.  By looking at the Frobenius action 
given in \cite{Ka1} (A2.2.6) of the restriction of this module to the cusp (which amounts to looking at the $q$-expansion), 
we see that in fact $\lambda = 1$ in a neighborhood of the cusp, hence globally due to the 
$q$-expansion principle. 

By a theorem due to Dwork (see \cite{Ka1} Theorem A2.3.6), there exists a Frobenius compatible splitting 
$$
\xymatrix{
	0 \ar[r] & \ul\omega \ar[r] & \wt\sH^\vee \ar[r] &
	\ul\omega^{-1} \ar@/_/[l]  \ar[r] & 0.
}
$$
The image $\wt u$ of the basis $\wt\eta$ is a horizontal section of $\wt\sH$, stable by the Frobenius $\Phi$.
The section $\wt u$ generates the unit root part $U$ of $\wt\sH$.  Using this basis, we see that the Frobenius on 
$\wt\sH$ acts as
$$
	\Phi \begin{pmatrix} \wt\omega \\   \wt u \end{pmatrix}  =
	\begin{pmatrix}   p  & 0  \\  0   &  1 \end{pmatrix}
	\begin{pmatrix} \wt\omega \\   \wt u \end{pmatrix}.
$$

We denote by $\xi$ the differential form in $\Omega^1_{\wt\cM/\bbZ_p}$ which corresponds to 
$\wt\omega^{\otimes 2}$ through the Kodaira-Spencer isomorphism 
$
	 \ul\omega^{\otimes 2} \cong \Omega^1_{\wt\cM/\bbZ_p}.
$
This $\xi$ is the dual basis
of the differential operator $N \theta$ given above.
If we express the Gauss-Manin connection using the basis $\{ \wt\omega, \wt u \}$, we have
$$
	\nabla \begin{pmatrix} \wt\omega \\   \wt u \end{pmatrix}  =
	A
	\begin{pmatrix} \wt\omega \\   \wt u \end{pmatrix} \otimes \xi
$$
for some $2\times2$-matrix $A$ whose components are in $V(\bbQ_p, \Gamma(N))$.  Then we see by looking
near the cusps that
$$
	A = \begin{pmatrix} 0 & 1 \\ 0 & 0\end{pmatrix}.
$$
Again by the $q$-expansion principle, this holds globally.  Hence we have 
$\nabla(\wt\omega) = \wt u \otimes \xi$ and $\nabla(\wt u) = 0$.
The dual basis $\wt\omega^{\vee}$, $\wt u^\vee$ of $\wt\omega$, $\wt u$ gives a basis of $\wt\sH$, and the connection is given by
$\nabla(\wt\omega^{\vee}) = 0$ and $\nabla(\wt u^\vee) = \wt\omega^\vee$.

%
\subsection{Calculation of Cohomology}\label{section: 4-3}
%

We now give a proof of Lemma \ref{lem: ord vanish}.

\begin{proof}[Proof of Lemma \ref{lem: ord vanish}]
	Consider a class $\alpha \in H^0_\rig(\sM^\ord, \Sym^k \sH(1))$.  Since $M^\ord$ is affine, by definition of 
	rigid cohomology, it is represented by a section
	$$
		\alpha \in \Gamma(M^\ord_{\bbQ_p}, j^\dagger \Sym^k \sH)
	$$
	such that $\nabla(\alpha) = 0$, where $M^\ord_{\bbQ_p}$ is the rigid analytic space associated to $M^\ord_{\bbQ_p}$.  
	If we let $\cM^\ord$ be the formal completion of $M^\ord$ with respect to the special fiber and 
	$\cM^\ord_{\bbQ_p}$ the rigid analytic space over ${\bbQ_p}$ associated to $\cM^\ord$, then we may regard $\alpha$ 
	as an element in $\Gamma(\cM^\ord_{\bbQ_p},  \Sym^k \sH)$ through the natural injection
	$$
		 \Gamma(M^\ord_{\bbQ_p}, j^\dagger \Sym^k \sH) \hookrightarrow  \Gamma(\cM^\ord_{\bbQ_p},  \Sym^k \sH).
	$$
	Furthermore, $\wt\cM_{\bbQ_p}$ is defined over $\cM^\ord_{\bbQ_p}$, and we have a commutative diagram
	$$
		\begin{CD}
			 \Gamma(\cM^\ord_{\bbQ_p},  \Sym^k \sH) @>{\subset}>>   \Gamma(\wt\cM_{\bbQ_p},  \Sym^k \wt \sH)   \\
			 @V{\nabla}VV	@V{\nabla}VV\\
		 	 \Gamma(\cM^\ord_{\bbQ_p},  \Sym^k \sH \otimes \Omega^1_{\cM^\ord_{\bbQ_p}})
			@>{\subset}>>  \Gamma(\wt\cM_{\bbQ_p},  \Sym^k \wt\sH \otimes \Omega^1_{\wt\cM_{\bbQ_p}}).
		\end{CD}	
	$$
	By consideration of the previous section, the module $\Sym^k \wt\sH$ has a basis consisting of
	$\wt\omega^{\vee n} \wt u^{\vee k-n}$ for $0 \leq n \leq k$.  If we denote by $\wt\alpha$ the image of $\alpha$
	in $\Gamma(\wt\cM_{\bbQ_p},  \Sym^k \wt \sH)$, then it is of the form
	$$
		\wt \alpha = \sum_{n=0}^k c_n \wt\omega^{\vee n} \wt u^{\vee k-n}
	$$
	for some functions $c_n \in V(\bbQ_p, \Gamma(N))$.  Since $\nabla(\alpha) = 0$, we have $\nabla(\wt\alpha) = 0$.
	Hence we have
	\begin{multline*}
		\nabla(\wt \alpha) =\left(  \sum_{n=1}^{k}  (k-n+1)  c_{n-1} \wt\omega^{\vee n} \wt u^{\vee k-n} 
		+  \sum_{n=0}^k N \theta(c_n) \wt\omega^{\vee n} \wt u^{\vee k-n} \right) \otimes \xi \\
		=0.
	\end{multline*}
	This gives the differential equations
	$
		N \theta(c_0) = 0
	$
	and
	$$
		N \theta(c_{n}) = - (k-n +1) c_{n-1}
	$$
	for $1 \leq n \leq k$.  By \eqref{eq: N theta}, the differential operator $N \theta$ acts as $q (d/dq)$ on the $q$-expansion.
	Hence the fact that $N \theta(c_0) = 0$ implies that $c_0$ is constant.  Furthermore, since the constant term of 
	$N \theta(c_1)$ with respect to the $q$-expansion must be zero,  the equation $N \theta(c_1) = - k c_0$
 	implies that both sides of this equation must be zero.  Hence we see that $c_0 = 0$ and $c_1$ is constant.  
	By continuing this argument for $1 \leq n \leq k$, we see that $c_0 = c_1 = \cdots = c_{k-1} = 0$ and $c_k$ is constant.  
	Hence we have
	$$
		\wt \alpha = c_k \wt\omega^{\vee k}
	$$
	for some constant $c_k \in {\bbQ_p}$.  Finally, since $\wt \alpha$ is the image of an element 
	$\alpha \in \Gamma(\cM^\ord_{\bbQ_p}, \Sym^k \sH)$, it must be invariant under the action of $[a,1]$ for any
	$a \in \bbZ_p^\times$.  Hence
	$$
		[a,1]^*  \wt \alpha = [a,1]^* (c_k  \wt\omega^{\vee k}) = a^{-k} c_k  \wt\omega^{\vee k} = \wt\alpha = c_k  \wt\omega^{\vee k}
	$$
	for any $a \in \bbZ_p^\times$.  This implies that $c_k = 0$, hence $\wt \alpha = 0$.
	This proves that $\alpha = 0$ as desired.
\end{proof}

%
%
%
\section{$p$-adic Eisenstein series and the syntomic class}
%
%
%

In this section, we introduce the $p$-adic Eisenstein series and prove our main theorem.
We first start with a review of $p$-adic modular forms.

%
\subsection{$p$-adic modular forms}
%

In this section, we define a modified version of Katz measure which will be used to construct
$p$-adic Eisenstein series of negative weights.  We first review the definition of $p$-adic modular 
forms $\Phi_{k,r,f}$ defined by Katz.

\begin{definition}
	We define the $p$-adic modular form $\Phi_{k,r,f}$ as the $p$-adic modular form in 
	$V(\bbZ_p, \Gamma(N)^\arith) \otimes \bbQ_p$ defined in \cite{Ka3} Lemma 5.11.4.
\end{definition}

By definition, we have
\begin{align*}
	\Phi_{k,0,f} &= G_{k+1,0,f},  &
	\Phi_{0,r,f} &= G_{0, r+1, f^t}.
\end{align*}
By \cite{Ka3} Lemma 5.11.0, Definition 5.11.2 and Lemma 5.11.4, 
this function is known to satisfy the $q$-expansion
\begin{multline*}
	2\Phi_{k,0,f} =
	L(-k, f(m,0) - (-1)^{k} f(-m,0)) \\
	+   \sum_{n>0}q^n  \sum_{dd'|n} \left( d^k  f(d,d') - (-d)^k  f(-d,-d') \right)
\end{multline*}
for $k \geq 2$ and
$$
	2\Phi_{k,r,f} = \sum_{n>0}q^n  \sum_{dd'|n} \left( d^r (d')^k f(d,d') - (-d)^r (-d')^k f(-d,-d') \right)
$$
if $r$, $k \geq 1$.

\begin{proposition}\label{pro: mu k}
	We fix an integer $k > 0$.  For $r \geq 0$ and functions $f: (\bbZ/N)^2 \rightarrow \bbZ_p$,
	we let
	$$
		\Phi^{(p)}_{k+1,r,f} :=  \Phi_{k+1,r, f(u,v)} - p^{r} \Frob(\Phi_{k+1, r, f(u,pv)}).
	$$
	Then there exists  a measure $\mu_N^{k+1}$ on $\bbZ_p \times (\bbZ/N)^2$ 
	whose moments are given by
	\begin{equation}\label{eq: moments}
		\int_{\bbZ_p \times (\bbZ/N)^2} y^r f(u,v) d \mu_N^{k+1} = 2 \Phi^{(p)}_{k+1, r, f} 
	\end{equation}
	for any $r \geq 0$ and $f: (\bbZ/N)^2 \rightarrow \bbZ_p$.
\end{proposition}

\begin{proof}
	We use the integrality criterion for $p$-adic measures \cite{Ka3} Lemma 6.0.9.
	By calculation of the $q$-expansion and our choice of $k$, 
	the constant term of  $\Phi_{k+1,r, f}$ is \text{zero} unless $r=0$.  Again by calculation of the $q$-expansion,
	we see that the constant term of $\Phi_{k+1,r,f(u,v)}$ is equal to the constant term of $\Frob(\Phi_{k+1,r, f(u,pv)})$,
	which implies that the constant term of $\Phi^{(p)}_{k+1,r,f}$
	is zero for any $r \geq 0$.    As in the proof of \cite{Ka3} Theorem 6.1.1,
	the integrality of the other terms of the $q$-expansion may be seen as follows.
	If we write $\binom{y}{r} = \sum_{m=0}^r c(n,r) y^m$, then the $q$-expansion of
	$$
		\sum_{m=0}^r c(m,r) 2\Phi^{(p)}_{k+1,m, f}
	$$
	is given by
	\begin{multline*}
		 \sum_{n>0}q^n  \sum_{dd'|n} \left( d^{k+1} \binom{d'}{r}  f(d,d') -  (-d)^{k+1} \binom{-d'}{r} f(-d,-d') \right)\\
		 -  \sum_{n>0}q^{pn}  \sum_{dd'|n} \left(  d^{k+1} \binom{pd'}{r} f(d,pd') - (-d)^{k+1}\binom{-pd'}{r}  f(-d,-pd') \right).
	\end{multline*}
	Hence we see that 
	$
		\sum_{m=0}^r c(m,r) 2\Phi^{(p)}_{m, k+1, f} 
	$ is integral in $V(\bbZ_p, \Gamma(N)^\arith)$.  By \cite{Ka3} Lemma 6.0.9,
	this implies that
	\eqref{eq: moments} defines a $p$-adic measure on $\bbZ_p \times (\bbZ/N)^2$ with values
	in $V(\bbZ_p, \Gamma(N)^\arith)$.  
\end{proof}

\begin{remark}
	Let $(a,b)$ be an element in $G(N) := \bbZ_p^\times \times (\bbZ/N)^\times$.
	In \cite{Ka3} Theorem 6.1.1, Katz defined a $p$-adic measure $\mu^{(a,b)}_N$ on $\bbZ_p^2 \times (\bbZ/N)^2$
	satisfying the interpolation property
	$$
		\int_{\bbZ_p^2 \times (\bbZ/N)^2} x^k y^r d \mu_N^{(a,b)}  = 2 \Phi_{k,r,f} - 2 [a,b] \Phi_{k,r,f},
	$$ 
	where $[a,b]$ denotes the action of $G(N)$ on $V(\bbZ_p, \Gamma(N)^\arith)$ given in \cite{Ka3} 5.3.2.
	The relation of our measure $\mu_N^{k+1}$ to $\mu_N^{(a,b)}$ is given by the formula
	$$
		(1 - [a,b]) \int_{\bbZ_p \times (\bbZ/N)^2} \psi(y) d \mu_N^{k+1}  =	
		\int_{\bbZ_p \times \bbZ_p^\times \times (\bbZ/N)^2} x^{k+1} \psi(y) d \mu_N^{(a,b)}.
	$$
\end{remark}

%
\subsection{Eisenstein series of negative weight}
%

Using the measure $\mu_N^{k+1}$ defined in the previous section, we define the $p$-adic Eisenstein
series of negative weight.  The following result is important in defining such Eisenstein series.

\begin{lemma}
	The measure $\mu_N^{k+1}$ defined in Proposition \ref{pro: mu k} has support on $\bbZ_p^\times \times (\bbZ/N)^2$.
\end{lemma}

\begin{proof}
	We prove that 
	$$
		\int_{p \bbZ_p \times (\bbZ/N)^2} \psi(y) f(u,v) d \mu_N^{k+1}  =0
	$$
	for any continuous function $\psi : \bbZ_p \rightarrow \bbZ_p$ and $f : (\bbZ/N)^2 \rightarrow \bbZ_p$.
	By continuity, the $q$-expansion of $\int_{\bbZ_p \times (\bbZ/N)^2}  \psi(y) f(u,v) d \mu_N^{k+1}$
	is given by
	\begin{multline*}
		\sum_{n>0}q^n  \sum_{dd'|n} \left( d^{k+1} \psi(d')  f(d,d') - (-d)^{k+1} \psi(-d') f(-d,-d') \right) \\
		- \sum_{n>0}q^{pn}  \sum_{dd'|n} \left( d^{k+1} \psi(pd')  f(d,pd') - (-d)^{k+1}\psi(-pd')  f(-d,-pd') \right)
	\end{multline*}
	Note that we have
	$$
		\int_{p \bbZ_p \times (\bbZ/N)^2} \psi(y) f(u,v) d \mu_N^{k+1}  =
		\frac{1}{p} \sum_{\zeta_p \in \mu_p} \int_{\bbZ_p \times (\bbZ/N)^2} \zeta^y_p \psi(y) f(u,v) d \mu_N^{k+1},
	$$
	where $\zeta_p$ is a primitive $p$-th root of unity.  By applying the $q$-expansion formula to the function 
	$\widetilde{\psi}(y) = \zeta_p^y \psi(y)$ and noting that $\sum_{\zeta_p} \zeta_p^y = p$ is $p|y$ and $=0$ 
	otherwise, we see by calculating the $q$-expansion that the right hand side of the above equality
	is zero.  Hence we have our assertion.
\end{proof}

Using the above fact, we define the $p$-adic Eisenstein series $\Phi^{(p)}_{k+1,r, f}$ for $k>0$ and $r< 0$ as follows.

\begin{definition}
	For any integers $k>0$ and $r \in \bbZ$, we define the $p$-adic Eisenstein series $\Phi^{(p)}_{k+1,r, f}$ to be
	the $p$-adic modular form such that
	$$
		2 \Phi^{(p)}_{k+1,r, f} := \int_{\bbZ_p^\times \times (\bbZ/N)^2} y^r f(u,v) d \mu_N^{k+1}
	$$
	in $V(\bbZ_p, \Gamma(N)^\arith)$.
\end{definition}

Recall Lemma \ref{lem: comparison Katz} that the Eisenstein series $E_{k+2,0,\varphi}$ is related to $G_{k,0,f}$ through the formula
$$
	E_{k+2,0,\varphi}(\tau,g)=G_{k+2,0,P_1(\wh{g\varphi})}(\tau).
$$
We define a $p$-adic version $E^{(p)}_{k+2,r,\varphi}$ as follows.

\begin{definition}
	For any integers $k$, $r$ such that $k>0$ and $\varphi : (\bbZ/N)^2 \rightarrow \bbZ_p$,
	we define the \textit{$p$-adic Eisenstein-Kronecker series} $E^{(p)}_{k+2, r, \varphi}$ to be the
	$p$-adic modular form in $V(\bbZ_p, \Gamma(N))$ given by
	$$
		E^{(p)}_{k+2, r, \varphi}(g): = \Phi^{(p)}_{k+1, r, P_1(\widehat{g \varphi})} \,  \in  V(\bbZ_p, \Garith(N))
	$$
	on the component for $g \in GL_2(\bbZ/N)$.
\end{definition}

From the definition, we have the following.

\begin{lemma}
	For any integer $k>0$, we have
	$$
		E^{(p)}_{k+2, 0, \varphi} = (1 - \frobphi^*_{\wt\cM}) E_{k+2,0, \varphi},
	$$
	where $\frobphi_{\wt\cM} := \Frob \otimes \sigma$ is the Frobenius on $\wt\cM$ of \eqref{eq: Frob full level}.
\end{lemma}

\begin{proof}
	We have
	$
		\sigma(g \varphi)(m,n) = [p] g \varphi(m,n) =g \varphi( p^{-1} m,n ).
	$
	Hence
	\begin{multline*}
		\sigma(P_1 (g \varphi))(m,n) =
		\frac{1}{N}\sum_{v}  \sigma(g \varphi)(v,n) \exp \left[ \frac{2 \pi i v m}{N} \right] \\
		= 	\frac{1}{N}\sum_{v}  g \varphi(p^{-1} v,n) \exp \left[ \frac{2 \pi i v m}{N} \right]\\
		= 	\frac{1}{N}\sum_{v'} g \varphi(v',n) \exp \left[ \frac{2 \pi i v' (pm)}{N} \right]
		= P_1(g \varphi)(pm, n),
	\end{multline*}
	where we have used the change of variables $v = p v'$ in $\bbZ/N$.
	Since $P_1(\widehat{g\varphi}) = (P_1(g\varphi))^t$, we have
	$$
		\sigma(P_1(\widehat{g\varphi}))(m,n) = P_1(\widehat{g\varphi})(m, pn).
	$$
	This implies that
	\begin{multline*}
		\Phi^{(p)}_{k+1, 0, P_1(\widehat{g\varphi})} := \Phi_{k+1, 0, P_1(\widehat{g\varphi})(u,v)} - 
		\Frob \left( \Phi_{k+1, 0, P_1(\widehat{g\varphi})(u, pv)} \right)  \\
		= (1 - \frobphi^*)  \Phi_{k+1, 0, P_1(\wh{g\varphi})}.
	\end{multline*}
	Our assertion now follows from the fact that
	$E^{(p)}_{k+2, 0, \varphi}(g)  = \Phi^{(p)}_{k+1, 0, P_1(\widehat{g\varphi})}$
	and	$E_{k+2, 0, \varphi}(g) = G_{k+2,0, P_1(\wh{g\varphi})} = \Phi_{k+1, 0, P_1(\wh{g\varphi})}$.
\end{proof}

\begin{lemma}
	Suppose $k$ is an integer $>0$.
	The $p$-adic Eisenstein-Kronecker series $E^{(p)}_{k+2,r,\varphi}$ satisfy the differential equation
	$$
		\left(  q \frac{d}{dq} \right) E^{(p)}_{k+2,r,\varphi} = E^{(p)}_{k+2+1,r+1,\varphi}.
	$$
\end{lemma}

\begin{proof}
	By continuity, the $q$-expansion of $2 E^{(p)}_{k+2,r,\varphi}(g) = 2 \Phi^{(p)}_{k+1, r, P_1(\wh{g\varphi})}$ is
	{\small\begin{multline*}
		\sum_{n > 0} q^n \sum_{dd'|n} \left( d^{k+1} d'^r P_1(\wh{g\varphi})(d,d') - (-d)^{k+1} (-d')^r P_1(\wh{g\varphi})(-d,-d')  \right) \\
		- \sum_{n > 0} q^{pn} \sum_{dd'|n} \left( d^{k+1} (pd')^r P_1(\wh{g\varphi})(d,pd') - (-d)^{k+1} (-pd')^r 
		P_1(\wh{g\varphi})(-d,-pd') \right).
	\end{multline*}}
	Our assertion follows by direct calculation.
\end{proof}

%
\subsection{The syntomic class}
%

We next determine the section 
$$
	\alpha_\Eis^{k+2}(\varphi) \in \Gamma(\ol\cM_{\bbQ_p},  \Sym^k \sH_\rig)
$$
giving the syntomic Eisenstein class $\Eis^{k+2}_\syn(\varphi)$.  By definition,
$\alpha^{k+2}_\Eis(\varphi)$ is an element satisfying the differential equation
\begin{equation}\label{eq: diff alpha}
	\nabla(\alpha^{k+2}_\Eis(\varphi)) = (1 - \Phi) \Eis^{k+2}_\dR(\varphi).
\end{equation}	
In order to describe $\alpha^{k+2}_\Eis(\varphi)$ explicitly, we consider its image with respect to the natural injection
$$
	\Gamma(\ol\cM_{\bbQ_p}, \Sym^k \sH_\rig):=\Gamma(\ol\cM_{\bbQ_p}, j^\dagger \Sym^k \sH) \hookrightarrow \Gamma(\widetilde{\cM}_{\bbQ_p}, \Sym^k \wt\sH).
$$

\begin{definition}
	We define the element $\wt\alpha^{k+2}_\Eis(\varphi) \in \Gamma(\widetilde{\cM}_{\bbQ_p}, \Sym^k \wt\sH)$ by the formula
	\begin{equation}\label{eq: def alpha 00}
		\wt\alpha^{k+2}_\Eis(\varphi)  = \sum_{n= 0}^{k}  \frac{(-1)^n}{(k-n)!} 
		E^{(p)}_{k+1-n, -1-n,\varphi} \wt\omega^{\vee n} \wt u^{\vee k-n},
	\end{equation}
	where $E^{(p)}_{k+1-n, -1-n, \varphi}$ are the $p$-adic Eisenstein-Kronecker series.
\end{definition}

The connection on $\ol\cM_{\bbQ_p}$ and $\widetilde{\cM}_{\bbQ_p}$ gives the commutative diagram
$$
	\xymatrix{%
			\Gamma(\ol\cM_{\bbQ_p},  \Sym^k \sH_\rig) \ar[r] \ar[d]_\nabla & 
				\Gamma(\widetilde{\cM}_{\bbQ_p}, \Sym^k \wt\sH) \ar[d]_\nabla  \\
			\Gamma(\ol\cM_{\bbQ_p}, \Sym^k \sH_\rig \otimes \Omega^1_{\ol\cM_{\bbQ_p}})\ar[r] & 
				\Gamma(\widetilde{\cM}_{\bbQ_p}, \Sym^k \wt\sH \otimes \Omega^1_{\widetilde{\cM}_{\bbQ_p}}).
	}%
$$
By definition of the $p$-adic Eisenstein-Kronecker series, we have
\begin{multline*}
	\nabla(\widetilde\alpha^{k+2}_\Eis(\varphi)) =  \sum_{n= 0}^{k}  \frac{(-1)^n}{(k-n)!} 
	E^{(p)}_{k+2-n, -n,\varphi} \wt\omega^{\vee n} \wt u^{\vee k-n} \otimes \xi \\
	+  \sum_{n= 1}^{k}  \frac{(-1)^{n+1}}{(k-n)!} E^{(p)}_{k+2-n, -n,\varphi} \wt\omega^{\vee n} \wt u^{\vee k-n} \otimes \xi
	= 	\frac{1}{k!} E^{(p)}_{k+2, 0,\varphi}  \wt u^{\vee k} \otimes \xi.
\end{multline*}	
Therefore, if we identify  $\Eis^{k+2}_\dR(\varphi)$ with its image in 
$\Gamma(\widetilde{\cM}_{\bbQ_p}, \Sym^k \sH\otimes \Omega^1_{\widetilde{\cM}_{\bbQ_p}})$, 
then by definition of $\Eis^{k+2}_\dR(\varphi)$, we have 
$$	
	\nabla(\widetilde\alpha^{k+2}_\Eis(\varphi)) = (1 - \wt\Phi) \Eis^{k+2}_\dR(\varphi),
$$
where $\wt\Phi$ is the Frobenius on $\Sym^k \wt \sH$.
Hence this element satisfies a condition similar to \eqref{eq: diff alpha}.  
We next prove that $\widetilde\alpha^{k+2}_\Eis(\varphi)$ is in fact the image of an element $\alpha^{k+2}_\Eis(\varphi)$ in 
$\Gamma(\ol\cM_{\bbQ_p}, \Sym^k \sH_\rig)$. 

\begin{lemma}
	There exists an element  
	$$
		\alpha^{k+2}_\Eis(\varphi) \in \Gamma(\ol\cM_{\bbQ_p}, \Sym^k \sH_\rig)
	$$ whose image in
	$\Gamma(\widetilde{\cM}_{\bbQ_p}, \Sym^k \wt\sH \otimes \Omega^1_{\widetilde{\cM}_{\bbQ_p}})$
	is $\wt\alpha^{k+2}_\Eis(\varphi)$.
\end{lemma}

\begin{proof}
	It is sufficient to prove that $\widetilde\alpha^{k+2}_\Eis(\varphi)$ descends to $\ol\cM_{\bbQ_p}$.  In order to prove this statement,
	it is sufficient to show that $\widetilde\alpha^{k+2}_{\Eis}(\varphi)$ is invariant under the action of $[a,1]$ for any 
	$a \in \bbZ_p^\times$.
	By definition, $[a,1]$ acts on $\wt\omega$ by $[a,1] \wt\omega = a^{-1} \wt\omega$, and since $\wt u = \wt\omega^{-1}$,
	we have $[a,1] \wt u = a \wt u$.  Hence by duality, we have
	$$
		[a,1] (\wt\omega^{\vee n} \wt u^{\vee k-n}) = a^{2n-k} (\wt\omega^{\vee n} \wt u^{\vee k-n}).
	$$
	By \cite{Ka3} Lemma 5.11.6, we have
	\begin{equation*}
		[a,1] E^{(p)}_{k+1-n,  -1-n,\varphi} = a^{k-2n} E^{(p)}_{k,r,\varphi}.
	\end{equation*}
	Our assertion now follows from the definition of $\widetilde\alpha^{k+2}_\Eis(\varphi)$.
\end{proof}

We may now use $\alpha^{k+2}_\Eis(\varphi)$ to explicitly describe the syntomic Eisenstein class.
Our result shows that the syntomic Eisenstein class is related to $p$-adic Eisenstein-Kronecker
series, much in the same way as the Eisenstein class in absolute Hodge cohomology is related to
real analytic Eisenstein-Kronecker series.

\begin{theorem}\label{thm: main}
	The syntomic Eisenstein class 
	$
		\Eis_\syn^{k+2}(\varphi) 
	$
	restricted to the ordinary locus is expressed as
	$$
		\Eis_\syn^{k+2}(\varphi) = (\alpha^{k+2}_{\Eis}(\varphi), \Eis_\dR^{k+2}(\varphi)),
	$$
	where $\Eis_\dR^{k+2}(\varphi)$ is the section of $\Gamma(\ol M_K, \Sym^k \sH \otimes \Omega^1_{\ol M_K}(\log \Cusp))$
	defined in \eqref{eq: de Rham Eisenstein} giving the de Rham Eisenstein class, and $\alpha^{k+2}_{\Eis}(\varphi)$
	is the section defined in the previous lemma which is the unique section mapping to
	$$
		\wt\alpha^{k+2}_\Eis(\varphi)  = \sum_{n= 0}^{k}  \frac{(-1)^n}{(k-n)!} 
		E^{(p)}_{k+1-n, -1-n,\varphi} \wt\omega^{\vee n} \wt u^{\vee k-n}
	$$
	on $\wt \cM_{\bbQ_p}$.
\end{theorem}

\begin{proof}
	By construction of $\alpha^{k+2}_\Eis(\varphi)$, we have
	$$
		\nabla(\alpha^{k+2}_\Eis(\varphi)) = (1 - \Phi) \Eis_\dR^{k+2}(\varphi).
	$$
	Furthermore, since $\ol M_K$ is a curve, we have $\nabla(\Eis_\dR^{k+2}(\varphi)) = 0$. 
	Hence by Proposition \ref{prop: fund class}, the pair $(\alpha^{k+2}_{\Eis}(\varphi), \Eis_\dR^{k+2}(\varphi))$
	defines an element in $H^1_\syn(\sM^\ord, \Sym^k \sH)$.  By Corollary \ref{cor: fund class}, we see that
	this class maps to the de Rham Eisenstein class through the boundary morphism.  Hence the characterization
	in Proposition \ref{pro: ord to dR} of the syntomic Eisenstein class on the ordinary locus shows that
	$$
		\Eis_\syn^{k+2}(\varphi) = (\alpha^{k+2}_{\Eis}(\varphi), \Eis_\dR^{k+2}(\varphi))
	$$
	as desired.
\end{proof}

\appendix

%
%
%
\section{Rigid syntomic cohomology}       
%
%
%

In this section, we review the basic facts concerning rigid syntomic cohomology.
Let $K$ be a finite unramified extension of $\bbQ_{p}$ with ring of integers $\cO_{K}$ and 
residue field $k$. We denote by $\sigma$ the lifting  of the absolute Frobenius of $k$ to $\cO_{K}$ and $K$.

%
\subsection{Filtered overconvergent $F$-isocrystal}		
%

Here, we define the notion of \textit{filtered overconvergent $F$-isocrystals}, which are the smooth coefficients for rigid syntomic cohomology.  This is what is referred to as \textit{syntomic coefficients} in \cite{Ba1} Definition 1.1, but extend to deal with the case 
without a global Frobenius.

\begin{definition}
	We say that a pair of schemes $\sX = (X, \ol X)$ is a \textit{smooth pair}, if $X$ is a smooth scheme of finite type over 
	$S := \Spec\, \cO_{K}$, and $\ol X$ is a smooth compactification such that the complement  $D:= \ol X \setminus X$ is 
	a simple normal crossing divisor relative to $S$.  
\end{definition}

In what follows, we fix a smooth pair $\sX = (X, \ol X)$.
Let $X_{k}: = X \otimes k$ and $\ol X_k := \ol X \otimes k$.  
We denote by $\Isocda(X_{k}/K)$ the category of overconvergent isocrystals on 
$X_{k}$ (\cite{Ber2} Definition 2.3.6).    The realization, in the sense of \cite{Ber2} p.68, of the category $\Isocda(X_{k}/K)$ 
may be given as follows.  Let $\cX$ and $\ol\cX$ be the formal completion of $X$ and $\ol X$ with respect to the special fiber, 
and let $\cX_{K}$ and $\ol\cX_K$ be the associated rigid analytic space.  Note that these rigid analytic spaces are 
the tubular neighborhoods
\begin{align*}
	] X_k [_{\ol\cX} &= \cX_K, &   ]\ol X_k[_{\ol\cX} &= \ol\cX_K.
\end{align*}

For any strict neighborhood $U$ of $j \colon\cX_{K} \hookrightarrow \ol\cX_K$, we let
$j^{\dagger}$ be the functor defined in \cite{Ber2} (2.1.1.3).  This functor associates to a coherent $\cO_{U}$-module $M$ the
coherent $j^{\dagger} \cO_{\ol\cX_K}$-module $j^{\dagger} M$.
The category $\Isocda(X_{k}/K)$ may be realized as the category whose objects consists of 
the pair $(M_{\rig}, \nabla_\rig)$, where $M_{\rig}$ is a coherent 
$j^{\dagger} \cO_{X_{K}^{\an}}$-modules on $\ol\cX_K$ with integrable connection
$$
	\nabla_\rig \colon M_{\rig} \rightarrow M_{\rig} \otimes \Omega^{1}_{\ol\cX_K}
$$
which is \textit{overconvergent}, in the sense of loc. cit. Definition 2.2.5.
We denote by $F^{*}_{\sigma}$ the functor defined in \cite{Ber2} 2.3.7
$$
	F^{*}_{\sigma} \colon\Isocda(X_{k}/K) \rightarrow \Isocda(X_{k}/K)
$$
obtained as the composition of the base extension $\sigma: K \rightarrow K$ with
the absolute Frobenius $F_k: X_k \rightarrow X_k$ of the special fiber.  	
A \textit{Frobenius structure} on an overconvergent isocrystal $\cM_\rig$ on $X_k$ is 
an isomorphism
$$
	\Phi: F_{\sigma}^{*} \cM_{\rig} \xrightarrow{\cong} \cM_{\rig}
$$
in $\Isocda(X_k/K)$.

Next, let $X_{K} := X \otimes K$ and $\ol X_{K} : = \ol X \otimes K$.  
Consider a coherent $M$ module on $\ol X_{K}$ with integrable connection
$$\nabla : M \rightarrow M \otimes \Omega^{1}_{\ol X_{K}}(\log D)$$
on $M$ with logarithmic singularities along $D$.  We may associate to $M$ an overconvergent
isocrystal $\cM_\rig$ on $X_k$ as follows.  Let $X_{K}^{\an}$ be the rigid analytic space associated to $\ol X_{K}$.  
Then $X_{K}^{\an}$ is a strict neighborhood of  $j \colon\cX_{K} \hookrightarrow \ol\cX_K$.  
We let $M_\rig$ be the $j^{\dagger} \cO_{\ol\cX_K}$-module
$$
	M_\rig := j^{\dagger} ( M|_{X_{K}^{\an}})
$$
with an overconvergent connection $\nabla_\rig$ induced from $\nabla$.  
Then $(M_\rig, \nabla_\rig)$ represents an overconvergent isocrystal $\cM_\rig$ in $\Isocda(X_k/K)$.
We now give the definition of the category of filtered 
overconvergent $F$-isocrystals on the smooth pair $\sX$.
\begin{definition}
	We define the category  $S(\sX)$ of \textit{filtered overconvergent $F$-isocrystals} on $\sX$
	to be the category consisting of the 4-uple $$\shM = (M, \nabla, F^{\bullet}, \Phi),$$ where
	\begin{enumerate}
		\item $M$ is a coherent $\cO_{\ol X_K}$-module with an integrable connection
		$$
			\nabla: M \rightarrow M \otimes \Omega^1_{\ol X_K}(\log D)
		$$
		with logarithmic singularities along $D$.
		\item $F^{\bullet}$ is a descending, exhaustive, and separated filtration on $M$ satisfying Griffiths transverality
		$$
			\nabla(F^{\bullet} M) \subset  F^{\bullet-1} M  \otimes \Omega^{1}_{\ol X_{K}}(\log D).
		$$
		\item  Let $\cM_\rig$ be an overconvergent isocrystal represented by $(M_\rig, \nabla_\rig)$.
		Then $\Phi$ is a Frobenius structure on $\cM_\rig$.
	\end{enumerate}
	The morphisms in this category are morphisms of coherent $\cO_{\ol X_K}$-modules compatible
	with the additional structures.
\end{definition}

Next, we define the de Rham and rigid cohomology of filtered overconvergent $F$-isocrystals.  
Let $\sX$ be a smooth pair, and let $\shM = (M, \nabla, F^{\bullet}, \Phi)$ be a filtered 
overconvergent $F$-isocrystal on $\sX$. 
Let
\begin{align*}
	\DR^\bullet_\dR(M)&:= M \otimes \Omega^{\bullet}_{\ol X_{K}}(\log D) \\
	\DR^\bullet_\rig(M_\rig) &:=M_\rig \otimes \Omega^\bullet_{\ol\cX_K}
\end{align*}
where $\Omega^{\bullet}_{\ol X_{K}}(\log D)$ is the de Rham complex on $\ol X_{K}$ with
logarithmic singularities along $D$.  Then $\DR^\bullet_\dR(M)$ has a filtration defined by
\begin{equation}\label{eq: filtration}
	F^{m} \DR^\bullet_\dR(M) := F^{m-q} M \otimes \Omega^{q}_{\ol X_{K}}(\log D).
\end{equation}
We associate to $M$ the de Rham cohomology
\begin{align*}
	H^{i}_{\dR}(\sX, \cM) &:=  R^{i} \Gamma(\ol X_{K}, \DR^\bullet_\dR(M) ),
\end{align*}
which has a Hodge filtration defined by the Hodge to de Rham spectral sequence
\begin{equation}\label{equation: Hodge to de Rham}
	E_{1}^{p,q} = H^{p}\left(\ol X_K, \Gr^{q}_{F}\left( \DR^\bullet_\dR(M) \right)\right) \Rightarrow
	H^{p+q}_{\dR}(\sX, \cM).
\end{equation}

Let $(M_\rig, \nabla_\rig)$ be the overconvergent $F$-isocrystal associated to $(M, \nabla)$. 
The rigid cohomology for $\cM$ is defined as
$$
	H^{i}_{\rig}(X_k, \cM) :=  R^{i} \Gamma(\ol\cX_K, \DR_\rig^\bullet(M_\rig)).
$$
This cohomology has a Frobenius $\Phi$ induced from the Frobenius $\Phi$ on $M_{\rig}$.
As in \cite{Ba1} Definition 1.12, we have a natural homomorphism
\begin{equation}\label{equation: comparison}
	H^{i}_{\dR}(\sX, \cM) \rightarrow H^{i}_{\rig}(X_k, \cM).
\end{equation}

\begin{definition}
	Let $\shM$ be a filtered overconvergent $F$-isocrystal on $\sX$.  We say that $\shM$ is 
	\textit{admissible}, if it satisfies the following conditions.
	\begin{enumerate}
		\item The spectral sequence \eqref{equation: Hodge to de Rham} degenerates at $E_{1}$.
		\item The morphism \eqref{equation: comparison} is an isomorphism of $K$-vector spaces.
		\item The $K$-vector space \eqref{equation: comparison} with the Hodge filtration on de Rham cohomology
		and Frobenius on rigid cohomology is \textit{weakly admissible} in the sense of Fontaine.
	\end{enumerate}
\end{definition}

\begin{remark}
	The above definition of admissibility is ad hoc.  There should be a definition of 
	admissibility for filtered overconvergent $F$-isocrystals which would imply the above.
\end{remark}


\begin{definition}
	Suppose $\shM = (M, \nabla, F, \Phi)$ is an admissible filtered overconvergent $F$-isocrystal on 
	the smooth pair $\sX = (X, \ol X)$.
	We define the rigid cohomology $H^i_\rig(\sX, \shM)$ of 
	the smooth pair $\sX$ with coefficients in $\shM$ to be the $K$-vector space
	$$
		H^i_\rig(\sX, \shM) := H^i_\rig(X_k, \cM),
	$$
	with a natural Frobenius $\Phi$ induced from the Frobenius $\Phi$ on $M_\rig$
	and a Hodge filtration $F^\bullet$ induced from the Hodge filtration
	of $H^i_\dR(\sX, \cM)$ through the isomorphism \eqref{equation: comparison}.
\end{definition}

%
\subsection{Higher direct images}				
%

Next, we define the higher direct image of filtered overconvergent $F$-isocrystals for proper and smooth morphisms
between smooth pairs.
Let $\sX = (X, \ol X)$ and $\sY = (Y, \ol Y)$ be smooth pairs.  A morphism $u \colon\sX \rightarrow \sY$ between 
smooth pairs is a map $u \colon \ol X \rightarrow \ol Y$ such that $u(X) \subset Y$.

\begin{definition}
	We say that a map of smooth pairs $u \colon \sX \rightarrow \sY$ is \textit{proper}, if $u|_X$ is proper.
	We say that $u$ is \textit{smooth}, if $u|_X$ is smooth.
\end{definition}

In this subsection, we define the higher direct images of filtered overconvergent $F$-isocrystals for a
proper and smooth map $u \colon\sX \rightarrow \sY$.  In what follows, we assume that 
$u$ is proper and smooth.
Let $D = \ol X \setminus X$ and $D' = \ol Y \setminus Y$.  
We define the sheaf of relative logarithmic differential 
$\Omega^1_{\ol X/\ol Y, \log}$ as the cokernel
$$
	0 \rightarrow u^* \Omega^1_{\ol Y}(\log D') \rightarrow
	\Omega^1_{\ol X}(\log D) \rightarrow \Omega^1_{\ol X/\ol Y,\log}
	\rightarrow 0,
$$
and let 
$	
	\Omega^q_{\ol X/\ol Y,\log} = \wedge^q \Omega^1_{\ol X/\ol Y, \log}.
$
Suppose $M$ is a coherent $\cO_{\ol X_K}$-module with logarithmic connection
$$
	\nabla: M \rightarrow M \otimes_{\cO_{\ol X}} \Omega^1_{\ol X}(\log D).
$$
We define the relative de Rham complex 
\begin{equation*}
	\DR^\bullet_{X/Y}(M) := M \otimes_{\cO_{\ol X}} \Omega^\bullet_{\ol X/\ol Y, \log}.
\end{equation*}
Then the direct image for de Rham cohomology of $(M, \nabla)$ is defined to be the coherent $\cO_{\ol Y_K}$-module
$$
	R^q u_{\dR*} M := R^q u_* \DR^\bullet_{X/Y}(M),
$$
which has an integrable logarithmic connection, called the Gauss-Manin connection, defined as in \cite{Ka2} as follows.

We define a filtration on the de Rham complex $\DR^\bullet(M)$ by
$$
	\Fil^q \DR^\bullet(M) := \text{Image} \left[
	\DR^{\bullet -q}(M) \otimes_{\cO_{\ol X}} u^* \Omega^q_{\ol Y}(\log D')
	\rightarrow \DR^\bullet(M)
	\right].
$$
Then we have
$$
	\gr_\Fil^q\DR^\bullet(M) = \DR^{\bullet-q}_{X/Y}(M) \otimes_{\cO_{\ol X}} u^{*} 
	\Omega^q_{\ol Y}(\log D').
$$
Then this filtration gives the spectral sequence for filtrations
$$
	E_1^{qr} =  R^{q+r} u_* \gr_\Fil^q\DR^\bullet(M) 
	\Rightarrow  R^{q+r} u_* \DR^\bullet(M),
$$
where the $E_1$-term may be written as
\begin{align*}
	E_1^{qr} &=  R^{q+r} u_* \left( \DR^{\bullet-q}_{X/Y}(M) \otimes_{\cO_{\ol X}} u^{*} 
	\Omega^q_{\ol Y}(\log D') \right)\\
	&=  R^r u_{\dR *} M  \otimes_{\cO_{\ol X}} u^{*} 
	\Omega^q_{\ol Y}(\log D').
\end{align*}
The Gauss-Manin connection
$$
	\nabla \colon R^r u_{\dR*} M \rightarrow  R^r u_{\dR*} M \otimes \Omega^1_{\ol Y}(\log D')
$$ 
is defined as the connecting morphism
$
	d_{1}^{0r} : E_1^{0r} \rightarrow E_{1}^{1r}
$
of the above spectral sequence.

The higher direct image for rigid cohomology may be define using a similar method.  Note that since $u|_X$ is
smooth, the map $u \colon\ol\cX_K \rightarrow \ol\cY_K$ is smooth in a neighborhood of $X_k$.  Let
$j^\dagger \Omega^1_{\ol \cX_K/\ol \cY_K}$ be the relative de Rham differential on $\ol\cX_K$.  Consider 
an overconvergent isocrystal $\cM_\rig$ on $X_k$ realized as $(M_\rig, \nabla_\rig)$.
Then the relative de Rham complex associated to this realization is
$$
	\DR^\bullet_{X/Y}(M_\rig) := M_\rig \otimes_{j^\dagger \cO_{\ol\cX_K}} j^\dagger \Omega^1_{\ol\cX_K/\ol\cY_K}.
$$
Then the higher direct image for rigid cohomology is defined to be the module
$$
	R^q u_{\rig*}\cM_\rig := R^q u_* \DR^\bullet_{X/Y}(M_\rig),
$$
which by \cite{Ber1} Theorem 5 is a coherent $j^\dagger \cO_{\ol\cY_K}$-module with an integrable
overconvergent connection.  See \cite{Ts} \S 3.2 for a detailed construction of this Gauss-Manin connection. 

Suppose $\cM_\rig = (M_\rig, \nabla_\rig)$ has a Frobenius structure $\Phi : F_\sigma^* \cM_\rig \xrightarrow\cong \cM_\rig$
compatible with the connection.   The pull-back by the absolute Frobenius $F_k \colon\ol X_k \rightarrow \ol X_k$ 
induces a base change morphism
$$
	F_\sigma^* \left(  R^q u_{\rig*} M_\rig \right) \rightarrow  R^q u_{\rig*} \left( F_{\sigma}^* M_\rig \right),
$$
which is an horizontal isomorphism by \cite{Ts} Proposition 2.3.1.  Composed
with $\Phi$, we have a Frobenius structure
\begin{equation}\label{equation: define Frobenius}
	\Phi: F_\sigma^* \left(  R^q u_{\rig*} M_\rig \right) \xrightarrow\cong  R^q u_{\rig*} M_\rig.
\end{equation}

Suppose $u \colon\sX \rightarrow \sY$ is a proper smooth morphism of smooth pairs, 
and let $(M, \nabla)$ be a coherent module on $\ol X_K$ with integrable connection
with logarithmic poles along $D$.  Then for $M_\rig := j^\dagger(M|_{X_K^\an})$, 
we have the following.

\begin{proposition}
	There exists a canonical isomorphism
	\begin{equation}\label{equation: RD isom}
		j^\dagger \left(  (R^q u_{\dR*} M)|_{Y_K^\an} \right) \xrightarrow\cong  R^q u_{\rig*} M_\rig.
	\end{equation}
\end{proposition}

\begin{proof}
	Consider the commutative diagram
	$$
		\xymatrix{
			X_K^\an \ar[r] \ar[d]_{u^\an_K} & X_K \ar[d]_{u_K} \ar[r]&  \ar[d]^{u_K}\ol X_K \\
			Y_K^\an \ar[r] & Y_K \ar[r]& \ol Y_K.
		}
	$$
	The flat base change for the second square gives an isomorphism
	$$
		 (R^q u_{K*} M)|_{Y_K} \cong  R^q u_{K*}(M|_{X_K}).
	$$
	Combined with the base change for the first square, we have a homomorphism
	\begin{equation}
		 (R^q u_{K*}M)|_{Y^\an_K} \rightarrow   R^q u^\an_{K*}(M|_{X_K^\an})
	\end{equation}
	which is an isomorphism by GAGA from the assumption that $u|_{X}$ is proper.
	Since the map $u \colon\ol\cX_K \rightarrow \ol\cY_K$ is quasi-compact and quasi-separated, 
	cohomological functors and direct limits commute (See \cite{Ts} \S 4.1.1 for details).  
	Hence we have an isomorphism
	\begin{equation}
		j^\dagger  R^q u^\an_{K*}(M|_{X_K^\an}) \cong  R^q u^\an_{K*}(M_\rig).
	\end{equation}
	Our assertion now follows by composing the above isomorphisms.
\end{proof}

\begin{definition}\label{definition: HDI}
	Let $u\colon\sX \rightarrow \sY$ be a proper smooth morphism of smooth pairs, and let $\cM := (M, \nabla, F^\bullet, \Phi)$
	be a filtered overconvergent $F$-isocrystal on $\sX$. We define the higher direct image $ R^q u_* \cM$ by
	$$
		 R^q u_* \cM := ( R^q u_{\dR*} M, \nabla, F^\bullet, \Phi),
	$$
	where $\nabla$ is the Gauss-Manin connection, $F^\bullet$ is the filtration on 
	$ R^q u_{\dR*}(M)$ induced from the Hodge filtration on $M$ and $\Phi$ is the Frobenius 
	induced through \eqref{equation: RD isom} from the Frobenius \eqref{equation: define Frobenius} 
	on $ R^q u_{\rig*} M_\rig$.
\end{definition}

%
\subsection{Rigid syntomic cohomology}			
%

In this section, we well recall the theory of rigid syntomic cohomology with coefficients of \cite{Ba1}, 
with slight modification to allow for the case without a global Frobenius.  
We first define the notion of an overconvergent Frobenius for a smooth pair.

\begin{definition}\label{def: overconvergent Frobenius}
	Let $\sX = (X, \ol X)$ be a smooth pair.
	Then an overconvergent Frobenius $\frobphi_X = (\phi, \frobphi_V)$ on $\sX$ is a pair of morphisms such that
	$\phi:  \cX \rightarrow \cX$ is a morphism of $\sV$-formal schemes lifting
	the absolute Frobenius $F_k$ of $X_k$, and 
	$
		\frobphi_V \colon V \rightarrow \ol \cX_K
	$
	is a morphism of rigid analytic spaces on some strict  neighborhood $V$ of  $\cX_K$ in $\ol \cX_K$
	which extends $\frobphi_K := \phi \otimes K$.  In other words, we have a commutative diagram
	$$
		\xymatrix{
			X_k \ar[r] \ar[d]^{F_k} & \cX \ar[d]^\phi & \ar[l]  \cX_K \ar[d]^{\frobphi_K} \ar@{^{(}->}[r]<-0.5ex> & V 
			\ar[dl]^{\frobphi_V} \ar@{^{(}->}[r]<-0.5ex> & \ol\cX_K \\
			X_k \ar[r] & \ol\cX &  \ar[l] \ol\cX_K. &
		}
	$$
\end{definition}

\begin{remark}\label{rem: Frobenius lifting}
	\begin{enumerate}
		\item 
		In our previous paper \cite{Ba1}, we assumed the existence of a global Frobenius
		$\frobphi_X : \ol\cX \rightarrow \ol\cX$ on the entire formal scheme $\ol\cX$.
		This would naturally give rise to $\frobphi_X$ in our sense.  This weak form is necessary to consider the theory
		when $X$ is a modular curve.
		\item
		If $X$ is an affine smooth scheme $X = \Spec\,A$, then by a theorem 
		of van der Put \cite{vdP} (2.4), there exists
		a Frobenius $\phi: A^{\dagger} \rightarrow A^{\dagger}$ on the weak completion of $A$ lifting the absolute 
		Frobenius of the special fiber.  This combined with \cite{Ber1} (2.5.3) shows that an overconvergent Frobenius
		exists in this case.
	\end{enumerate}
\end{remark}

\begin{definition}
	We denote by $\frD_K$ the category of syntomic data on $K$ defined as follows.
	The object in this category is a pair $(\sX, \frobphi_X)$, where
	\begin{enumerate}
		\item $\sX = (X, \ol X)$ is a smooth pair.
		\item $\frobphi_X = (\phi, \frobphi_V)$ is an overconvergent Frobenius on $\sX$.
	\end{enumerate}
	A morphism between syntomic datum $(\sX, \frobphi_X)$, $(\sY, \frobphi_Y)$ in $\frD_K$ is a morphism of smooth pairs
	$
		u \colon \sX \rightarrow \sY
	$
	 compatible with the Frobenius.
\end{definition}

We will often omit the $\frobphi_X$ from the notation and simply write $\sX$ for $(\sX, \frobphi_X)$.
In what follows, we fix a syntomic data $\sX$.
Suppose $\cM_\rig$  is an overconvergent isocrystal in $\Isocda(X_k/K)$ represented by $ (M_\rig, \nabla_\rig)$.
Then by \cite{Ber1} Proposition 2.5.5, the overconvergent isocrystal $F_\sigma^* \cM_\rig$ is
expressed as the pull-back by $\frobphi_V$ of $(M_\rig, \nabla_\rig)$.  Hence a Frobenius structure 
$\Phi$ on $\cM_\rig$ may be explicitly realized as an isomorphism $\Phi$ of $j^\dagger \cO_{\ol\cX_K}$-modules
$$
	\Phi:  \frobphi_V^* M_\rig \xrightarrow\cong M_\rig
$$
on $\ol\cX_K$, horizontal with respect to the connection $\nabla_\rig$.  Using this realization, we may define 
rigid syntomic cohomology of $(\sX, \frobphi_X)$ with coefficients in an admissible overconvergent 
$F$-isocrystal $\cM$ essentially following the method of \cite{Ba1}.

Let $I$ be a finite set, and let $\frU = \{ \ol U_i \}_{i \in I}$ be a covering of $\ol X$ by Zariski open sets. 
We put $\ol U_{i_0 \cdots i_n K} = \cap_{0 \leq j \leq n} \ol U_{i_j K}$.  Next, let $U_i = \ol U_i \cap X$, and
let $\cU_{i K}$ be the rigid analytic space over $K$ associated to the formal completion $\cU_i$ of $U_i$ 
with respect to the special fiber.  For $\cU_{i_0 \cdots i_n K} = \cap_{0 \leq j \leq n} \cU_{i_j K}$, we denote by
$j_{i_0 \cdots i_n}$ the inclusion
$$
	j_{i_0 \cdots i_n} \colon \cU_{i_0 \cdots i_n K} \hookrightarrow \ol\cX_K.
$$
We let $R^\bullet_\dR(\frU, \cM)$ be the simple complex associated to the Cech complex
$$
	\prod_i \Gamma(\ol U_{i K}, \DR_\dR^\bullet(M)) \rightarrow
		\prod_{i_0, i_1} \Gamma(\ol U_{i_0 i_1 K}, \DR_\dR^\bullet(M))
		\rightarrow \cdots
$$
and we let
$R^\bullet_\rig(\frU, \cM)$ be the simple complex associated to 
$$
	\prod_i \Gamma(\ol \cX_K, j_i^\dagger\DR_\rig^\bullet(M_\rig)) \rightarrow
		\prod_{i_0, i_1} \Gamma(\ol \cX_K, j_{i_0 i_1}^\dagger \DR_\rig^\bullet(M_\rig))
		\rightarrow \cdots.
$$
The complex $R^\bullet_\dR(\frU, \cM)$ has a filtration induced from the Hodge filtration, and there are 
canonical homomorphisms
\begin{align*}
	\frobphi_\frU &: K \otimes_{\sigma, K} R^\bullet_\rig(\frU, \cM)\rightarrow R^\bullet_\rig(\frU, \cM), \\
	\theta_\frU& : R_\dR^\bullet(\frU, \cM) \rightarrow R^\bullet_\rig(\frU, \cM)
\end{align*}
where the first morphism is induced from $\Phi$ and the overconvergent Frobenius $\frobphi_X$,
and the second from $\theta$.  We let
$$
	R^\bullet_\syn(\frU, \cM) := \Cone( F^0 R^\bullet_\dR(\frU, \cM) \rightarrow R^\bullet_\rig(\frU, \cM) )[-1],
$$
where the morphism is $(1 - \frobphi_\frU) \circ \theta_\frU$.

\begin{definition}
	We define the rigid syntomic cohomology of $\sX$ with coefficients in $\cM$ by
	$$
		H^m_\syn(\sX, \cM) := \varinjlim_{\frU} H^m( R^\bullet_\syn(\frU, \cM) ),
	$$
	where the limit is taken with respect to coverings $\frU$ ordered by refinements.
\end{definition}

Note that we have an canonical isomorphism
$$
	H^m( R^\bullet_\syn(\frU, \cM) ) \xrightarrow\cong H^m_\syn(\sX, \cM) 
$$
if the covering $\frU$ consists of affine open sets.  

\begin{proposition}
	By definition, we have a long exact sequence
	\begin{multline*}
		\cdots \rightarrow F^0 H^m_{\dR}(\sX, \cM) \xrightarrow{1 - \phi}
		H^m_\rig(X_k, \cM) \rightarrow H^{m+1}_\syn( \sX, \cM) \rightarrow \cdots
	\end{multline*}
	In the special case $\sV = (\Spec\, \,\cO_K, \Spec\, \, \cO_K)$ with Frobenius $\sigma$, then $S(\sV)$ is
	simply the category of filtered Frobenius modules.  For $\cM = (M, 0, F, \Phi)$ in $S(\sV)$, we have 
	\begin{align*}
		H^{0}_\syn( \sV, \cM) &= \Ker \left( F^0 M \xrightarrow{1-\Phi} M \right) \\
		H^1_\syn(\sV, \cM) &= \Coker \left( F^0 M \xrightarrow{1-\Phi} M \right)
	\end{align*}	
	and $H^m_\syn(\sV, \cM) = 0$ if $m \not= 0, 1$.
\end{proposition}

\begin{corollary}
	We have a short exact sequence
	\begin{multline*}
		0 \rightarrow H^1_\syn(\sV, H^m_\rig(\sX, \cM)) \rightarrow H^{m+1}_\syn(\sX, \cM) \\
		\rightarrow H^0_\syn(\sV, H^{m+1}_\rig(\sX, \cM)) \rightarrow 0.
	\end{multline*}
\end{corollary}

\begin{definition}
	We define the boundary map
	\begin{equation}\label{eq: leray boundary}
		H^{m}_\syn(\sX, \cM) \rightarrow H^m_{\dR}(\sX, \cM)
	\end{equation}
	to be the composition of the surjection 
	$$
		H^{m+1}_\syn(\sX, \cM) \rightarrow H^0_\syn(\sV, H^{m+1}_\rig(\sX, \cM))
	$$ 
	with the natural injection
	$$
		H^0_\syn(\sV, H^{m+1}_\rig(\sX, \cM)) \hookrightarrow H^m_{\dR}(\sX, \cM).
	$$
\end{definition}

%
\subsection{Cohomology class in $H^1$}				
%

In this section, we give a method to explicitly describe a cohomology class in the first syntomic cohomology of
an admissible filtered overconvergent $F$-isocrystal.  Suppose $\sX = (X, \ol X, \frobphi_X)$ is a syntomic data
and suppose $\sM = (M, \nabla, F, \Phi)$ is an admissible filtered overconvergent  $F$-isocrystal in $S(\sX)$.   
Then we have the following.

\begin{proposition}\label{prop: fund class}
	Suppose $\cM = (M, \nabla, F, \Phi)$ is an admissible filtered overconvergent $F$-isocrystal on $\sX$
	such that $F^0 M = 0$.	
	Then a class
	 $$
	 	[\alpha] \in H^1_\syn(\sX, \cM)
	$$ 
	is given uniquely by pairs of sections $(\alpha, \xi)$ for
	\begin{align*}
		\alpha  &\in \Gamma(\ol\cX_K, M_\rig),  &  \xi \in \Gamma(\ol X_K, F^{-1} M \otimes \Omega^1_{\ol X_K}(\log D))
	\end{align*}
	satisfying the conditions
	$
		\nabla(\alpha) = (1 - \Phi) \xi
	$
	and $\nabla(\xi) = 0$.
\end{proposition}

\begin{proof}
	We fix an affine open covering $\frU = \{ \ol U_i \}$ of $\ol X$.  Then we have
	$$
		H^1_\syn(\sX, \cM) = H^1( R^\bullet_\syn(\frU, \cM) ).
	$$
	The condition on the Hodge filtration of $\cM$ implies that 
	$ R^0_\syn(\frU, \cM) =  F^0 R^0_\dR(\frU, \cM) = 0$, and
	$F^0 R^\bullet_\dR(\frU, \cM)$ is given by the Cech complex
	$$ \left[  \prod_i \Gamma(\ol U_{iK}, F^0 \DR^1_\dR(M))
		\rightarrow   \prod_{i_0 i_1} \Gamma(\ol U_{i_0 i_1K}, F^0 \DR^1_\dR(M)) \rightarrow \cdots \right][1]
	$$
	for $F^0 \DR^1_\dR(M) = F^0 M \otimes \Omega^1_{X_K}(\log D)$.
	Suppose we have a class $[\alpha] \in H^1_\syn(\sX, \cM)$.  Then this class
	is represented by a pair
	$$
		(\alpha_\frU, \xi_\frU) \in R^{0}_\rig(\frU, \cM) \bigoplus F^0 R^1_\dR(\frU, \cM) 
	$$
	satisfying the cocycle conditions $\partial(\alpha_\frU) = (1 - \Phi) \xi_\frU$ and $\partial(\xi_\frU) = 0$,
	where $\partial$ is the differential operator on each of the complexes $R^{\bullet}_\rig$ and $ R^\bullet_\dR$.
	This representation is unique, since $R^0_\syn(\frU, \cM) = 0$ and thus there are no coboundaries.
	If we write $\alpha_\frU = (\alpha_i) \in \bigoplus_{i\in I} \Gamma(\ol\cX_K, j_i^\dagger M_\rig)$ and
	$$
		\xi_\frU = (\xi_i) \in \bigoplus_{i\in I} \Gamma(\ol U_{iK}, F^0 \DR^1_\dR(M)),
	$$
	then the cocycle conditions are $\nabla(\alpha_i) = (1-\Phi) \xi_i$, $\alpha_j = \alpha_i$ and $\xi_j = \xi _i$.
	Hence both $(\alpha_i)$ and $(\xi_j)$ paste together uniquely to global sections
	$\alpha \in \Gamma(\ol\cX_K, M_\rig)$ and $\xi \in \Gamma(\ol X_K, F^{-1} M \otimes \Omega^1_{\ol X_K}(\log D))$
	satisfying the differential equations $\nabla(\alpha) = (1 - \Phi) \xi$ and $\nabla(\xi) = 0$ as desired.
	Conversely, we see directly from the definition that a pair $(\alpha, \xi)$ satisfying the above conditions defines
	a class in $H^1_\syn(\sX, \cM)$.
\end{proof}

Suppose $\xi \in \Gamma(\ol X_K, M \otimes \Omega^1_{\ol X_K}(\log D))$ is an
element satisfying $\nabla(\xi) = 0$.  Then this defines a de Rham cohomology class
$$
	[\xi] \in H^1_\dR(\sX, \cM).
$$
By the previous proposition and the construction of the boundary morphism, we have the following.

\begin{corollary}\label{cor: fund class}
	Suppose $[\alpha] \in H^1_\syn(\sX, \cM)$ is of the form
	$$
		[\alpha] = (\alpha, \xi)
	$$
	as in the previous proposition, where $\alpha \in \Gamma(\ol\cX_K, M_\rig)$ and
	$\xi \in \Gamma(\ol X_K, M \otimes \Omega^1_{\ol X_K}(\log D))$.
	Then the image of $[\alpha]$ with respect to the boundary morphism
	$$
		H^1_\syn(\sX, \cM) \rightarrow H^1_\dR(\sX, \cM)
	$$
	is given by $[\xi]$.
\end{corollary}


\end{document}